\documentclass[11pt]{amsart}

\title{Recognising perfect fits}
\author{\textsc{Layne Hall}
}
\usepackage{amssymb}
\usepackage{amsmath}
\usepackage{amsthm}
\usepackage{mathtools}
\usepackage[font={small}]{caption}
\usepackage{subcaption}
\usepackage{algorithmicx}
\usepackage[noend]{algpseudocode}
\usepackage{algorithm}

\usepackage[hidelinks]{hyperref}
\hypersetup{
    colorlinks=true,
    linkcolor=[rgb]{.8,.2,.4},
    citecolor=[rgb]{.08,.7,.4},
    urlcolor=magenta
    }
\usepackage{enumerate}
\usepackage{enumitem}
\setlist[enumerate,1]{label=\arabic*.,ref=\arabic*}
\usepackage{textcomp}
\usepackage{titlesec}
\usepackage{nameref}
\usepackage{lipsum}
\usepackage[pagewise]{lineno} 
\usepackage[final]{microtype} 
\usepackage[capitalise]{cleveref}

\theoremstyle{plain}

\usepackage{thmtools}
\newtheorem{thm}{Theorem}[section]
\newtheorem{prop}[thm]{Proposition}
\newtheorem{lemma}[thm]{Lemma}

\newtheorem{coro}[thm]{Corollary}

\makeatletter
\let\c@algorithm\relax 
\makeatother

\usepackage{aliascnt}
\newaliascnt{algorithm}{thm} 

\theoremstyle{definition}
\newtheorem{algo/}[thm]{Algorithm}
\newtheorem{con/}[thm]{Construction}
\newtheorem{question/}[thm]{Question}
\newtheorem{example/}[thm]{Example}
\newtheorem{rem/}[thm]{Remark}
\newtheorem{defn/}[thm]{Definition}
\newtheorem{observation/}[thm]{Remark}
\newtheorem{notation/}[thm]{Notation}
\newtheorem{convention/}[thm]{Convention}
\newtheorem{process/}[thm]{Process}
\newtheorem{procedure/}[thm]{Procedure}
\newtheorem{sketch/}[thm]{Sketch}
\newtheorem{algothm/}[thm]{Algorithm}
\newtheorem{obs/}[thm]{Observation}

\newenvironment{obs}
  {%
   \pushQED{\qed}\begin{obs/}}
  {\popQED\end{obs/}}

\newenvironment{defn}
  {%
   \pushQED{\qed}\begin{defn/}}
  {\popQED\end{defn/}}

\newenvironment{example}
  {%
   \pushQED{\qed}\begin{example/}}
  {\popQED\end{example/}}

\newenvironment{question}
  {%
   \pushQED{\qed}\begin{question/}}
  {\popQED\end{question/}}
  
\newenvironment{rem}
  {%
   \pushQED{\qed}\begin{rem/}}
  {\popQED\end{rem/}}

\newenvironment{con}
  {%
   \pushQED{\qed}\begin{con/}}
  {\popQED\end{con/}}

\newcommand{\sing}{\{s_i\}}

\newcommand{\sspan}{\mathrm{Span}^s}
\newcommand{\uspan}{\mathrm{Span}^u}
\newcommand{\boxes}{\mathcal{B}}
\newcommand{\cs}{\mathrm{Stone}}
\newcommand{\ints}{\mathrm{Joints}}
\newcommand{\edges}{\mathrm{Edges}}
\newcommand{\faces}{\mathrm{Faces}}
\newcommand{\tets}{\mathrm{Tets}}
\newcommand{\cusps}{\mathrm{Cusps}}

\newcommand{\floworbs}{(\varphi,\sing)}

\DeclareMathOperator{\bbR}{\mathbb{R}}
\DeclareMathOperator{\bbN}{\mathbb{N}}
\DeclareMathOperator{\bbC}{\mathbb{C}}
\DeclareMathOperator{\bbZ}{\mathbb{Z}}

\numberwithin{algorithm}{section}
\numberwithin{equation}{section}

\titleformat{\title}
{\normalfont\scshape\centering}{\title}{1em}{}
\titleformat{\section}
{\normalfont\scshape\centering}{\thesection}{1em}{}
\titleformat{\subsection}
{\normalfont\bfseries}{\thesubsection}{1 em}{}

\usepackage[
backend=biber,
style=alphabetic,
citestyle=alphabetic,
maxbibnames=99
]{biblatex}
\begin{document}
\begin{abstract}
        A pseudo-Anosov flow is said to have perfect fits if there are stable and unstable leaves that are asymptotic in the universal cover. We give an algorithm to decide, given a box decomposition of a pseudo-Anosov flow, if the flow has perfect fits. As a corollary, we obtain an algorithm to decide whether two flows without perfect fits are orbit equivalent.
\end{abstract}
\maketitle

\section{Introduction}
Given a three-manifold $M$, there are rich interactions between the topology of $M$ and the dynamics of pseudo-Anosov flows on $M$ \cite{fried-sections, thurston-surface-bundles}. Many of these relationships depend on the homotopy properties of periodic orbits \cite{fenley-annals, fenley-qg, ss-link, bfm-classi}. When there are homotopies between periodic orbits, the flow must have \emph{perfect fits}. We give an algorithm to decide if a pseudo-Anosov flow has perfect fits.

A key motivation is the correspondence between flows and veering triangulations \cite{ss-link}. Let $\varphi$ be a pseudo-Anosov flow on a closed orientable three-manifold $M$. Suppose that $\varphi$ has no perfect fits. With $\sing$ the set of singular orbits of $\varphi$, Agol-Gu\'eritaud \cite{ag-talk, gueritaud-veering} construct a \emph{veering triangulation} of the drilled manifold $M^{\circ} = M-\sing$. Both Agol-Tsang \cite{agol-tsang} and Schleimer-Segerman \cite{ss-pair} give constructions from a veering triangulation to a pseudo-Anosov flow without perfect fits. This has established a correspondence between the two theories which has since been used to explore properties of the flow and the underlying manifold \cite{parlak-taut-alex, landry-minsky-taylor, tsang-birkhoff}.

\begin{figure}[ht]
    \centering
    \includegraphics[width=0.8\textwidth]{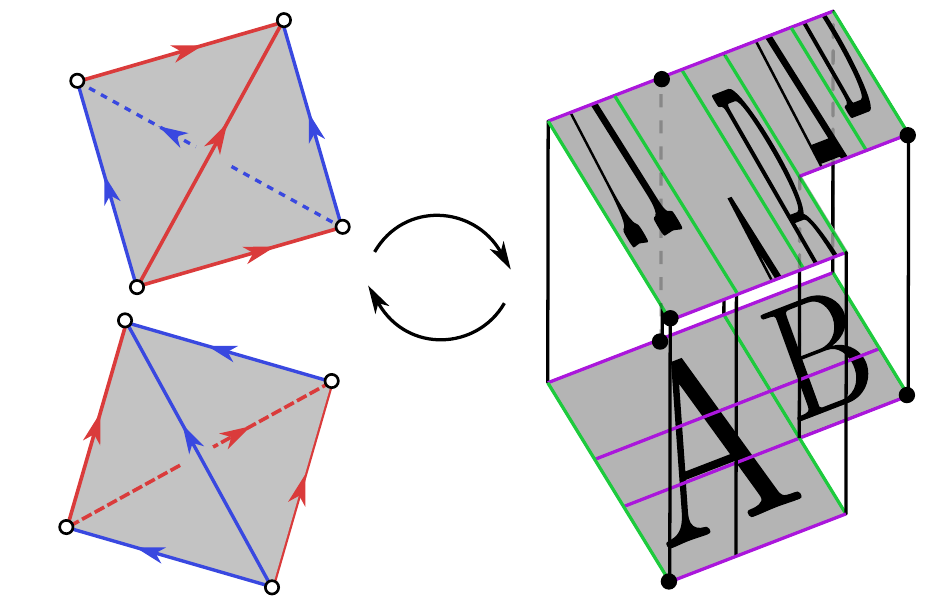}
    \caption{The main goal of the paper: to decide when we can pass between a veering triangulation (left) and a box decomposition for a pseudo-Anosov flow (right).}
    \label{fig:enter-label}
\end{figure}

An arbitrary pseudo-Anosov flow can be presented in the combinatorial form of a \emph{box decomposition} \cite[\S 3]{mosher1996laminations}. While box decompositions describe all flows, they are flexible, and can be subdivided arbitrarily. In contrast, when a veering triangulation does exist, it is a canonical invariant of the flow. This rigidity has been used to study flows computationally \cite{veering-census, Veering}.

Inspired by the rich computational theory for three-manifolds \cite{haken-normal, matveev, kuperberg}, we consider pseudo-Anosov flows from an algorithmic perspective. A natural problem is deciding when we can pass between a box decomposition and a veering triangulation. Our main result is solving this for pseudo-Anosov flows. The case of Anosov flows is fundamentally distinct from the pseudo-Anosov case, since there are not obvious orbits to drill. These are not addressed in this version of the theorem.
\begin{restatable*}{thm}{mainthm}
\label{thm:main}
    There is an algorithm that, given a box decomposition of a pseudo-Anosov flow, determines whether or not the flow has perfect fits.
\end{restatable*}
Anosov flows are instead covered in a generalised version of the main theorem, (\Cref{thm:main-ext}), which determines whether drilling specified orbits will produce a veering triangulation.

Given a decomposition for a three-manifold into flow boxes, deciding if this represents a pseudo-Anosov flow is straightforward (\Cref{rem:box-metric}), so \Cref{thm:main} addresses the main difficulty of the problem. Moreover, transitive pseudo-Anosov flows can be described by other discrete data, such as a veering triangulation with appropriate filling slopes. This can be converted to a box decomposition (\cite{agol-tsang, ss-pair}), to which we can apply the theorem.

The algorithm of \Cref{thm:main}, \texttt{HasPerfectFits}, consists of two routines that verify either side of the decision. We briefly sketch each side of the algorithm. The routine to verify that $\varphi$ has perfect fits, \texttt{FindFit}, is an application of two established results: a characterisation of perfect fits in terms of freely homotopic orbits due to Fenley \cite{fenley-qg}, and the solution to the conjugacy problem for closed three-manifold groups due to Sela \cite[Theorem 4.2]{sela} and Pr\'eaux \cite[Main Theorem]{preaux}. We enumerate periodic orbits using classical symbolic dynamics, then use the conjugacy problem to test Fenley's criterion.

For the routine in the other direction, \texttt{FindVeering}, we directly construct the veering triangulation $\mathcal{V}$ that canonically corresponds to the flow. This is done by iteratively building the universal cover of $M^{\circ}$ with boxes from $\mathcal{B}$ and applying the Agol-Gu\'eritaud construction. The routine terminates when we find a fundamental domain of the $\pi_1(M^{\circ})$-action in the universal cover $\tilde{\mathcal{V}}$ of the veering triangulation. So, the routine also returns the corresponding veering triangulation. Building this triangulation directly from a box decomposition takes the bulk of the paper.

With the resulting veering triangulation, we make progress towards the natural recognition problem for flows. Recall that two pseudo-Anosov flows are \emph{orbit equivalent} if there is a homeomorphism between the underlying manifolds that carries orbits to orbits. For a class of pseudo-Anosov flows, we say \emph{the orbit equivalence problem is solvable} if, given box decompositions of two flows in the class, there exists an algorithm to decide if the flows are orbit equivalent.

In the absence of perfect fits, the veering triangulation and filling data are enough to characterise the flow \cite{ss-final-step}.

\begin{restatable*}{coro}{equiv}\label{cor:equivalence}
The orbit equivalence problem for pseudo-Anosov flows without perfect fits is solvable.
\end{restatable*}

We can also use the veering triangulation to recognise if the underlying flow is a pseudo-Anosov suspension; this is \Cref{cor:suspension-recognition}.

\textbf{Outline.}
In \Cref{sec:prelims}, we give background on flows (\Cref{subsec:flows}), perfect fits (\Cref{subsec:pfits,subsec:homotopies}), and veering triangulations (\Cref{subsec:veering}). The routine for verifying that a flow has perfect fits (\Cref{rou:find_fit}) is given in \Cref{sec:find_fit}. In \Cref{sec:find_veering}, we detail the routine to verify that a flow does not have perfect fits (\Cref{rou:find_veering}). This includes a review of the Agol-Gu\'eritaud construction (\Cref{con:ag}) and a characterisation of the absence of perfect fits by veering triangulations (\Cref{lem:no-fits-iff-complete}). The main algorithm (\Cref{alg:has_fits}) of \Cref{thm:main} and that of \Cref{cor:equivalence} are presented in \Cref{sec:final-algos}. We conclude with a list of questions in \Cref{sec:qns}.\\

\textbf{Acknowledgements.} The author thanks Saul Schleimer for many interesting discussions about veering triangulations and pseudo-Anosov flows and for feedback on the paper. We thank Chi Cheuk Tsang for many helpful comments on an early version of the paper. Finally, we are very grateful to the anonymous reviewer; their detailed comments and corrections have greatly improved the paper.

The author is supported by the Warwick Mathematics Institute Centre for Doctoral Training, and gratefully acknowledges funding from the University of Warwick and the UK Engineering and Physical Sciences Research Council (Grant number: EP/W524645/1).

\section{Preliminaries}\label{sec:prelims}
\subsection{Flows and box decompositions}\label{subsec:flows}

        Pseudo-Anosov flows are described with two local models. The first model is a flow box; this models the flow at generic points. The second model is a pseudo-hyperbolic orbit; this models the flow at a set of exceptional orbits. We follow Mosher \cite[\S 3.1]{mosher1996laminations} in our definitions.
        
        \begin{defn}\label{def:flow box}
        A \emph{model flow box} $B$ is a labelled copy of the unit cube $I^s \times I^u \times I^t$, with each $I^* \simeq [0,1]$. We equip $B$ with a flow in the $t$-direction, and a pair of two-dimensional transverse foliations: a foliation $B^s$ whose leaves are rectangles with constant $u$-coordinate, and $B^u$ with constant $s$-coordinate. See \Cref{fig:flow box}.

    \begin{figure}
     \centering

     \begin{subfigure}[b]{0.4\textwidth}
         \centering
         \includegraphics[]{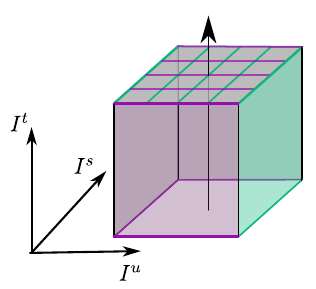}
        \caption{A flow box}
        \label{fig:flow box}
     \end{subfigure}
     \begin{subfigure}[b]{0.4\textwidth}
         \centering
         \includegraphics[scale=0.8]{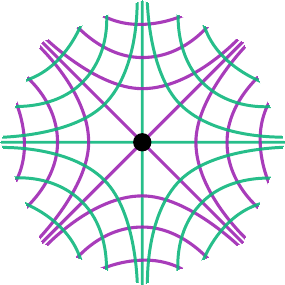}
            \caption{A singular orbit.
            }
            \label{fig:prong}
     \end{subfigure}
     
        \caption{}
        \label{fig:models}
\end{figure}
        Let $\mathrm{Top}(B) = I^s \times I^u \times \{1\}$ and $\mathrm{Bot}(B) = I^s \times I^u \times \{0\}$. Define a \emph{$u$-rectangle} of $\mathrm{Top}(B)$ to be a subset $H^s\times I^u  \times \{1\}$ for some interval $H^s\subset I^s$. By exchanging the first two coordinates, we define similarly a \emph{$s$-rectangle} of $\mathrm{Top}(B)$. Analogously, define both $u$-rectangles and $s$-rectangles of $\mathrm{Bot}(B)$.

        The \emph{vertical one-cells} of $B$ are the edges of $B$ tangent to the $t$ direction. A \emph{stable wall} is a side of a box of $B$ of the form $I^s\times \{u\} \times I^t$ for $u\in\{0,1\}$ minus the vertical one-cells.
        \end{defn}

        \begin{defn}\label{def:singular-orbit}
             For a hyperbolic matrix $A\in SL_2(\bbZ)$, the linear map $A:\bbR^2\to\bbR^2$ preserves a pair of foliations by lines on $\bbR^2$. Consider $\bbR^2$ as a copy of $\bbC$. For $k \geq 2$, let $f_k:\bbC\to\bbC$ be a lift of $A$ under the semi-branched cover $z\to z^{k/2}$ which fixes each component in the preimage of each quadrant of $\bbC$. Define $f_{k,n}$ as $f_k$ followed by a rotation by $2\pi n/k$ for $0\leq n < k$. This map preserves the preimage of the horizontal and vertical foliations.
            
            Suspend $f_{n,k}$ to obtain a flow $\psi_{n,k}$ on the resulting mapping torus of $\bbC$. The map $\psi_{n,k}$ preserves the pair of two-dimensional foliations corresponding to the preimage of the invariant foliations of $A$. We denote the foliation whose leaves are contracted as $F^s$ and the other as $F^u$. See \Cref{fig:prong}.
            
            Let $N_{k,n}$ be a small tubular neighbourhood of the periodic orbit of $\psi_{n,k}$ at the origin. Call $N_{k,n}$ a \emph{prong neighbourhood}.
        \end{defn}

        \begin{defn}
        \label{def:pa-flow}
        Let $M$ be a closed, oriented three-manifold. Let $\varphi$ be a continuous flow on $M$ with no stationary orbits. We call $\varphi$ a \emph{pseudo-Anosov flow} if there exists a pair of two-dimensional singular foliations denoted $M^s$ (stable) and $M^u$ (unstable) for which the following hold:
        \begin{enumerate}[itemsep=6pt]
            
            \item \label{def:item:non-empty-sing} There exists a \textbf{non-empty} set $\sing\subset M$ of \emph{singular} periodic orbits of $\varphi$ such that on $M-\{s_i\}$, the leaves of $M^s$ and $M^u$ are regular and transverse to each other.

            \item \label{def:item:singular} For each orbit $s_i\subset M$, there is a neighbourhood $X^i\subset M$ and a $n>2$ such that: there is an embedding into a prong neighbourhood $X^i\hookrightarrow N_{k,n}$ that preserves orbits and maps leaves of $M^s \cap X^i$ and $M^u \cap X^i$ into leaves of $F^s \cap N_{k,n}$ and $F^u \cap N_{k,n}$, respectively. Call such an orbit a \emph{$k$-prong} with \emph{rotation period} $k/\gcd(k,n)$.
            
            \item The pair $(\varphi, M)$ admits a \emph{flow box decomposition}: A cell decomposition $\mathcal{B}$ of $M$ into a collection $\{B_i\}$ of flow boxes, each equipped with an embedding $f_i:B \hookrightarrow M$ such that $f_i(B) = B_i$, satisfying:
            \begin{enumerate}[itemsep=5pt]
                \item Each orbit of $\sing$ lies in the two-skeleton of $\{B_i\}$
                \item The interiors of the boxes $\{B_i \}$ are pairwise disjoint
                
                \item The flow box embedding $f_i$ maps oriented flow-lines of $B$ into oriented flow-lines of $\varphi$, and leaves of the foliations $B^s, B^u$ into leaves of $M^s, M^u$, respectively.
    
                \item\label{def:item:markov} Let $\mathrm{Top}(B_i)$ and $\mathrm{Bot}(B_i)$ be the images of $\mathrm{Top}(B)$ and $\mathrm{Bot}(B)$ under the embedding $f_i$. The boxes must satisfy the \emph{Markov property}: if $\mathrm{Top}(B_i)\cap \mathrm{Bot}(B_j)$ is non-empty and not a point, the preimage of this intersection under $f_i$ is a collection of $s$-rectangles in $\mathrm{Top}(B)$. Likewise, the preimage under $f_j$ consists of $u$-rectangles in $\mathrm{Bot}(B)$.
            \end{enumerate}
        \label{def:item:box_decomp}
        
        \item Define the \emph{transition graph} $\mathcal{M}(\mathcal{B})\subset M$ by placing a vertex on the interior of each box $B_i$, and a directed edge passing through each rectangle in $\mathrm{Top}(B_i)\cap \mathrm{Bot}(B_j)$ from the vertex in $B_i$ to that in $B_j$. We require that $\mathcal{M}(\mathcal{B})$ has no circular sources or sinks.
        \label{def:item:hyperbolic} \qedhere
        \end{enumerate}
        
        \end{defn}

Compared to the literature, \Cref{def:pa-flow} excludes the case of an \emph{Anosov flow}, where the set of singular orbits $\sing$ is allowed to be empty. This is a natural setting for the paper because singular orbits are a canonical set of `marked orbits'. Our techniques apply in the following generalised setting.

\begin{defn}\label{def:marked-flow}
    Let $\varphi$ be an Anosov or pseudo-Anosov flow on $M$. Let $\sing\subset M$ be a non-empty collection of periodic orbits which contains all singular orbits of $\varphi$. We call $(\varphi,\sing)$ a \emph{(pseudo-)Anosov flow with marked orbits} on $M$. Let $\mathcal{B}$ be a flow box decomposition for $\varphi$. Suppose that the orbits $\sing$ are contained in the two-skeleton of $\mathcal{B}$. Then $\mathcal{B}$ is a \emph{box decomposition} of $(\varphi,\sing)$.
\end{defn}

Given a pseudo-Anosov flow $\varphi$ and $\sing$ the set of singular orbits, we have the naturally associated pseudo-Anosov flow with marked orbits $\floworbs$. A box decomposition for $\varphi$ is also naturally one for $\floworbs$.

\begin{rem}\label{rem:box-metric}
    \Cref{def:item:hyperbolic} of \Cref{def:pa-flow} can be exchanged for a familiar property: the existence of a path metric on $M$ for which the flow expands and contracts the invariant foliations, see \cite[Lemma 3.1.1]{mosher1996laminations}. Moreover, suppose that we are given a cell-decomposition $\mathcal{B}$ of $M$ into flow boxes. Mosher's argument shows that \Cref{def:item:box_decomp,def:item:hyperbolic}, along with mild conditions on the gluings between the sides of boxes (to ensure that flow-lines are glued to flow-lines and that no orbits are one-prongs), give a check-list to decide if $\mathcal{B}$ is a decomposition for a pseudo-Anosov flow.
\end{rem}

\begin{rem}\label{rem:smooth-case}
Recall that a flow is \emph{transitive} if it has a dense orbit. The given definition describes a \emph{topological} or even \emph{combinatorial} pseudo-Anosov flow as opposed to a \emph{smooth} such flow. See \cite[\S 5]{agol-tsang} for a discussion of why the notions are equivalent for transitive flows.

It is conjectured (\cite[\S 3.1]{mosher1996laminations}) that the classes are equivalent for non-transitive flows. However, the notion of perfect fits can be defined for either variant of pseudo-Anosov flow, and a non-transitive always has perfect fits, see \Cref{rem:toroidality}.
\end{rem}

\subsection{Perfect fits}\label{subsec:pfits}
Let $M$ be a closed oriented three-manifold. Let $\floworbs$ be a (pseudo-)Anosov flow on $M$ with marked orbits. The universal cover $\tilde{M}$ of $M$ is homeomorphic to $\bbR^3$ (a consequence of \cite[Theorem 6.1]{gabai-oertel}). Lift the flow $\varphi$ to a flow $\tilde{\varphi}$ on $\tilde{M}$. The \emph{orbit space} is the quotient $\mathcal{O}:=\tilde{M}/\tilde{\varphi}$. The orbit space is always homeomorphic to $\bbR^2$ \cite[Proposition 4.1]{fenley-mosher-qg}. Moreover, $\mathcal{O}$ inherits a natural action of $\pi_1(M)$. The invariant foliations of $\tilde{\varphi}$ quotient to invariant singular stable and unstable foliations $\mathcal{O}^u$, $\mathcal{O}^u$ of $\mathcal{O}$.

  Drill $M$ along the set of marked orbits to obtain $M^{\circ}:= M-\{s_i\}$. We have a corresponding drilled flow $\varphi^{\circ}$ on $M^{\circ}$. Let $N$ be the universal cover of $M^{\circ}$ and lift $\varphi^{\circ}$ to a flow $\psi$ on $N$. Quotient $N$ by $\psi$ to obtain the \emph{lifted orbit space} $\mathcal{P}$. The space $\mathcal{P}$ inherits stable and unstable foliations $\mathcal{P}^s$ and $\mathcal{P}^u$, and these are preserved by the lifted $\pi_1(M^{\circ})$-action. Since $\sing$ contains all singular orbits, puncturing removes the singularities and induces a pair of regular invariant foliations on $\mathcal{P}$ and $N$, which are then orientable.

Let $\mathcal{S}\subset\mathcal{O}$ be the discrete set of all lifts of $\{s_i\}$ to the universal cover. It is also useful to consider the \emph{punctured orbit space} $\mathcal{O}^{\circ}=\mathcal{O}-\mathcal{S}$. The lifted orbit space $\mathcal{P}$ is the universal cover of $\mathcal{O}^{\circ}$.

Let $r\subset\mathcal{O}$ be an open subset of $\mathcal{O}$. Equip $r$ with the restriction of the foliations $\mathcal{O}^{s}$ and $\mathcal{O}^u$. We say that $r$ is a \emph{rectangle} if:
\begin{itemize}
    \item there exists a homeomorphism $f_r:\mathrm{Int}(r)\to (0,1)^2$ which sends each leaf of $\mathcal{O}^s$ to a vertical segment and each leaf of $\mathcal{O}^u$ to a horizontal segment.
\end{itemize}
A compact rectangle has boundary lying along two stable and two unstable leaves. See \Cref{fig:pfits}.
\begin{figure}[ht]
            \centering
            \includegraphics[width=0.7\textwidth]{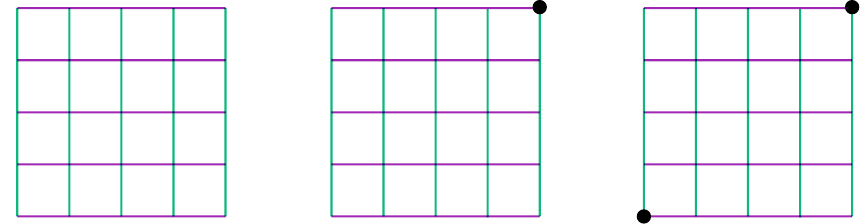}
            \caption{Left to right: a compact rectangle, a perfect fit rectangle, and a lozenge.}
            \label{fig:pfits}
\end{figure}
\begin{defn}
A rectangle $r\subset\mathcal{O}$ is a \emph{perfect fit rectangle} if $f_r$ extends to a homeomorphism $\bar{f}_r:r\to [0,1]^2- \{(1,1)\}$.
\end{defn}

The orbit space can be naturally compactified to a disk \cite[Theorem A]{fenley-cv-groups}, where a perfect fit rectangle is a rectangle with an ideal corner. In a perfect fit rectangle, there is a stable leaf $L\in M^s$ and an unstable leaf $K\in M^u$ that are coincident to the ideal corner. We say that $L$ and $K$ \emph{form a perfect fit}. A \emph{lozenge} is a rectangle $r$ such that $f_r$ can be extended to $\bar{f}_r: r\to [0,1]^2- \{(0,0),(1,1)\}$. See also \Cref{fig:pfits}. Truncating either of the ideal corners of the lozenges produces a perfect fit rectangle.

The definitions of rectangle and perfect fit rectangle generalises immediately to $\mathcal{O}^{\circ}$ and $\mathcal{P}$ in terms of their foliation pairs. A perfect fit rectangle in $\mathcal{O}^{\circ}$ arises either from a perfect fit rectangle in $\mathcal{O}$, or it can be introduced by puncturing $\mathcal{S}$. We now have the following generalised notion of perfect fits.
 \begin{defn}\label{def:rel_fits}
     Let $r\subset \mathcal{O}^{\circ}$ be a perfect fit rectangle. If $r$ includes into a compact rectangle in $\mathcal{O}$, we say that $r$ is a \emph{cusp rectangle}. Otherwise, we call $r$ a \emph{genuine perfect fit rectangle}. We extend this terminology to a perfect fit rectangle $r\subset \mathcal{P}$ according to whether its image in $\mathcal{O}^{\circ}$ is a cusp or genuine perfect fit rectangle.     
	
	 We say that \emph{$\floworbs$ has no genuine perfect fits} if every perfect fit rectangle in $\mathcal{P}$ (or $\mathcal{O}^{\circ}$) is a cusp rectangle. This is called \emph{no perfect fits relative to $\sing$} in \cite[Definition 5.13]{agol-tsang}.
 \end{defn}
 \begin{rem}\label{rem:fits_vs_genuinefits}
 Let $\varphi$ be a pseudo-Anosov flow and $\sing$ the set of singular orbits. Since singular orbits cannot appear on the interior of a lozenge, then $\varphi$ has no perfect fits if and only if the pair $\floworbs$ has no genuine perfect fits. 
 \end{rem}
 
\subsection{Homotopic periodic orbits}\label{subsec:homotopies}
Fenley characterises classical perfect fits by free homotopies. We say two periodic orbits $\alpha,\beta$ of $\varphi$ are \emph{almost freely homotopic} if $\alpha^n\simeq \beta^m$. Following \cite[\S 4]{fenley-annals}, we say that a periodic orbit $
\alpha$ is \emph{freely homotopic to itself} if there exists a non-trivial element $\eta \in \pi_1(M)$, which itself is not a multiple of $
\alpha$, that conjugates $\alpha$ to $\alpha^n$ for some $n$. Either of these phenomena implies the existence of points in $\mathcal{O}$ that are fixed by the action of a common element in $\pi_1(M)$. In \cite[Theorem 3.3]{fenley-chain} and \cite[Theorem 4.8]{fenley-good-geometry}, Fenley shows that if an element of $\pi_1(M)$ fixes distinct points in $\mathcal{O}$, there is a `chain of lozenges' connecting the points in $\mathcal{O}$. In particular, such orbits imply that $\varphi$ has perfect fits. Fenley later used the shadowing lemma to prove the converse \cite[Theorem B]{fenley-qg}.
\begin{thm}[Fenley]\label{thm:pfits-iff-homotopy}
    Let $\varphi$ be a pseudo-Anosov flow. Then $\varphi$ has perfect fits if and only if there is a pair of periodic points in $\mathcal{O}$ that are the corners of a lozenge. This is equivalent to the existence of either:
    \begin{itemize}
        \item a pair of almost freely homotopic periodic orbits of $\varphi$
        \item a periodic orbit of $\varphi$ that is freely homotopic to itself.  \qed
    \end{itemize}
    
\end{thm}
\begin{rem}\label{rem:precise-homotopies}
    We will be more precise about the free homotopies. Suppose that $\alpha$ and $\beta$ are periodic orbits (possibly the same) that form corners of a common lozenge in $\mathcal{P}$. So, we have $\alpha^n \simeq \beta^{m}$. Fenley shows in \cite[Theorem 1.2]{fenley-almost-homotopic} that when $\alpha$ and $\beta$ are regular orbits, $n$ and $m$ must have opposite signs, and we can take the powers $n$ and $m$ as having magnitude one or two. This is because the actions of the associated elements in $\pi_1(M)$ must fix the lozenge between the pair.
    
    Suppose now that $\alpha$ and $\beta$ are singular. The generalisation is that $n$ and $m$ can be taken as having magnitudes the rotation periods (\Cref{def:item:singular}, \Cref{def:pa-flow}) of the orbits $\alpha$ and $\beta$, respectively.
\end{rem}

We can obtain a similar characterisation of no genuine perfect fits for flows with marked orbits. First, we briefly recall generalised \emph{Dehn-Goodman-Fried surgery} \cite{goodman, fried-sections}.

\begin{con}[Goodman, Fried]\label{con:gf}
Fix $\gamma$ a periodic orbit of $\varphi$, and let $M^{\circ} = M-\{\gamma\}$. The stable leaf through $\gamma$ traces out a collection of disjoint curves on $\partial \overline{M^{\circ}}$. Let $s$ be a slope on $\partial \overline{M^{\circ}}$ that has algebraic intersection $k$ with these curves. For $\vert k\vert \geq 2$, Dehn filling the cusp with slope $s$ yields a pseudo-Anosov flow $\varphi(s)$ with a $\vert k\vert $-pronged singular orbit at the core of the filled torus. The periodic orbits of the new flow are the inclusion of the periodic orbits of $\varphi$ and the filled orbit. This construction is implicit in the work of Fried \cite{fried-sections}.
\end{con}

\begin{rem}\label{lem:pfits-iff-homotopy-general}
    Let $(\varphi,\sing)$ be a (pseudo-)Anosov flow with marked orbits. Let $\varphi(s)$ be a pseudo-Anosov flow with singular orbits filled into all of the marked orbits $\sing$. The two flows $\varphi$ and $\varphi(s)$ have the same lifted flow space $\mathcal{P}$, and one can observe (see \cite[Remark 3.8.]{lmt-transverse}) that $\floworbs$ has genuine perfect fits if and only if $\varphi(s)$ has perfect fits.
\end{rem}

Orbits $\alpha$ and $\beta$ are almost freely homotopic when $\alpha^n$ and $\beta^m$ are conjugate in $\pi_1(M)$. The integers $m$ and $n$ can be negative. The conjugacy problem is solvable in both the closed and with-boundary case by work of \cite[Theorem 4.2]{sela} and {Pr\`eaux} \cite[Main Theorem]{preaux}. These solutions are essential to both sides of the main algorithm.

\begin{rem}\label{rem:toroidality}
If a pseudo-Anosov flow $\varphi$ is non-transitive, the underlying manifold has an embedded transverse torus \cite[Proposition 1.1]{mosher-transitive}. Fenley shows that any pseudo-Anosov flow on a three-manifold $M$ with a $\bbZ^2$-subgroup of $\pi_1(M)$ is either an Anosov suspension or has perfect fits \cite[Main Theorem]{fenley-tori-give-fits}.
\end{rem}

One could use \Cref{rem:toroidality} to immediately decide whether a flow has perfect fits by checking if the underlying manifold is toroidal and not a torus bundle. We do not explicitly need \Cref{rem:toroidality} for our algorithm, since the perfect fits are detected regardless. However, it highlights that detecting perfect fits is only interesting when $M$ is atoroidal. Pseudo-Anosov suspensions give examples of flows without perfect fits on atoroidal manifolds \cite[Theorem G]{fenley-cv-groups}. We give some simple examples of pseudo-Anosov flows with perfect fits and singular orbits on atoroidal manifolds. Another way to construct such examples is by using the ideas presented in \cite[Comment 2, \S 4]{mosher-homology-II}. 

\begin{example}\label{eg:nontrivial}
    Let $\Sigma$ be a hyperbolic surface, and $\varphi$ the geodesic flow on the unit tangent bundle $T^1\Sigma$. Fix $\alpha_0\subset \Sigma$ an oriented simple closed geodesic. Choose $\gamma_0\subset \Sigma$ an oriented filling geodesic that intersects $\alpha_0$ in only one direction. Let $\alpha,\, \gamma\subset T^1\Sigma$ be the periodic orbits of $\varphi$ given by the canonical lifts of $\alpha_0, \gamma_0$, and $\bar{\alpha}\subset T^1\Sigma$ the lift of $\alpha_0^{-1}$. These lifts are periodic orbits of $\varphi$. Since $\gamma$ is filling, the complement $T^1\Sigma-\gamma$ is hyperbolic \cite[Appendix B]{foulon-hasselblatt}. We perform a large filling to {fill a singular orbit} into the cusp (\Cref{con:gf}) to obtain a pseudo-Anosov flow $\psi$ on a hyperbolic three-manifold $M$.
    
    In $T^1\Sigma$, the orbits $\alpha$ and $\bar{\alpha}$ bound two Birkhoff annuli, and by the choice of $\gamma_0$, one of these annuli is disjoint from $\gamma$. This annulus survives in $M$, and so the inclusions of $\alpha$ and $\bar{\alpha}$ into $M$ are freely homotopic orbits of $\psi$.
\end{example}

\subsection{Veering Triangulations}\label{subsec:veering}
Let $M$ be an oriented three-manifold with torus boundary. We define veering triangulations, first due to Agol \cite[Definition 4.1]{agol-veering}, following \cite[Definition 1.3]{hodgson-veering}.

        An \emph{ideal tetrahedron} $t$ is a tetrahedron with its vertices removed. We say that $t$ is \emph{taut} when equipped with:
        \begin{itemize}
            \item face co-orientations such that two faces are co-oriented into $t$ and two faces out of $t$
            
            \item an assignment of a dihedral angle of $0$ or $\pi$ to each edge of $t$ such that each vertex meets one $\pi$-angle edge and two $0$-angle edges. Moreover, an edge between two faces with the same co-orientation must have dihedral angle $\pi$. 
        \end{itemize}
        A \emph{veering tetrahedron} is a taut tetrahedron $t$ with a colouring of the edges in red and blue that satisfies the condition:
        \begin{itemize}
            \item For a vertex $v\in t$, let $e$ be the model edge with dihedral angle $\pi$ that meets $v$. Endow $\partial t$ an orientation by viewing $t$ from the outside of $t$. Suppose that $e, e', e''$ are the edges that meet $v$ in anti-clockwise order. Then $e'$ is coloured blue, and $e''$ is coloured red.
        \end{itemize}
        Viewing a veering tetrahedron from above, the four equatorial edges alternate colours, while the colours of the top and bottom edge are not determined, see \Cref{fig:taut-colours}.
        Now let $M$ be an oriented three-manifold with torus boundary. An \emph{ideal triangulation} (\cite[\S 4.2]{thurston-notes}) of $M$ is a collection of oriented model ideal tetrahedra $\mathcal{T}$ with orientation reversing face pairings, such that the realisation $\vert \mathcal{T}\vert $ is homeomorphic to $M$. An ideal triangulation $\mathcal{T}$ of $M$ is \emph{transverse taut} when each $t\in\mathcal{T}$ is endowed with a taut structure such that:
        \begin{itemize}
            \item two paired model faces have opposite co-orientations

            \item the sum of the dihedral angles for the collection model edges mapped to the same edge in $|\mathcal{T}|^{\circ}$ is $2\pi$.
        \end{itemize}
        
        Note that if a triangulation has edge colours, we can pull this back to an induced colouring on the edges of the individual tetrahedra.
        \begin{defn}\label{def:veering}
            Let $\mathcal{V}$ be a transverse taut ideal triangulation of an oriented three-manifold $M$ equipped with a colouring of the edges of $|\mathcal{V}|$ in red and blue. If the induced colouring on each $t\in\mathcal{V}$ makes $t$ a veering tetrahedron for all $t$, we say that $\mathcal{V}$ is a \emph{veering triangulation} of $M$.
        \end{defn}

        \begin{figure}[ht]
                \centering
                \includegraphics[]{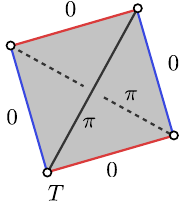}
                \caption{A taut, veering tetrahedron. Co-orientations on the faces are induced by the direction pointing out of the page.}
                \label{fig:taut-colours}
        \end{figure}
    We briefly summarise the key correspondence between veering triangulations and pseudo-Anosov flows:
    \begin{itemize}
        \item Flow to triangulation: Let $\varphi$ be a pseudo-Anosov flow. If $\varphi$ has no perfect fits, the drilled manifold $M^{\circ}$ admits a veering triangulation. This was announced by Agol-Gu\'eritaud \cite{ag-talk}; see \cite[\S 4]{landry-minsky-taylor} for an exposition. We use this explicitly in the form of \Cref{con:ag}.
        
        \item Triangulation to flow: Let $\mathcal{V}$ be a veering triangulation of a three-manifold with torus boundary $M^{\circ}$. Let $\mu = \{\mu_i\}$ be an assignment of filling slopes to each boundary component of $M^{\circ}$. Let $M^{\circ}(\mu)$ be the corresponding filled manifold. Off of finitely many lines of slopes for each $\mu_i$, there is a transitive pseudo-Anosov flow $\varphi$ on $M^{\circ}(\mu)$.
        
        Each filled torus has a periodic orbit of $\varphi$ at its core. With the filled cores included in the set of marked orbits, $\varphi$ has no perfect fits. The precise prong data of the filled orbits are determined by the combinatorics of $\mathcal{V}$ on $\partial M^{\circ}$; we defer to the actual constructions Agol-Tsang \cite[Theorem 5.1]{agol-tsang} and Schleimer-Segerman \cite[Theorem 10.1]{ss-pair} for details. The construction of Schleimer-Segerman is an inverse to that of Agol-Gu\'eritaud, which will be shown in an upcoming step of their program.
    \end{itemize}

In \Cref{subsec:agol-gueritaud}, we explain how the existence of a veering triangulation corresponding to the flow characterises the absence of perfect fits (\Cref{lem:no-fits-iff-complete}).

\section{Finding a perfect fit}\label{sec:find_fit}

Let $\mathcal{B}$ be a box decomposition for a pseudo-Anosov flow $\varphi$. In this section, we give an algorithm (\Cref{rou:find_fit}) that will be used to verify if $\varphi$ has perfect fits. We adapt this to the setting of a (pseudo-)Anosov flow with marked orbits and genuine perfect fits in \Cref{rem:anosov-case}.

\subsection{Enumerating periodic orbits}\label{subsec:orbits}
To generate homotopy classes of periodic orbits of $\varphi$, we use the classical idea (for example, \cite{BW-II, templates-book}) of encoding them up to homotopy with the transition graph $\mathcal{M}(\mathcal{B})$. We recall from \Cref{def:flow box} the notion of \emph{walls} and \emph{vertical one-cells} of the boxes in $\mathcal{B}$.

\begin{con}\label{con:itin}
Let $\gamma\subset M$ be a periodic orbit of $\varphi$. We define the \emph{itineraries} of $\gamma$. The construction has three cases:  when $\gamma$ lies inside the walls of $\mathcal{B}$, the vertical one-skeleton, or neither. We start with the case where $\gamma$ is disjoint from both the walls and the vertical one-skeleton. Choose, as a base point, any point along $\gamma$ that lies in the interior of a box in $\mathcal{B}$. We follow along $\gamma$, recording which rectangle $\gamma$ passes through when it crosses between boxes. This gives a periodic word. Let $w$ be the primitive part of this word; $w$ is unique up to cyclic permutation. We call $w$ the \emph{itinerary of $\gamma$}. Dual to the sequence of rectangles $w$ records are edges in $\mathcal{M}(\mathcal{B})$; their union is a loop $\gamma_w$.

Now suppose that $\gamma$ meets the two-skeleton. The Markov property implies that $\gamma$ is contained in the union of a circular tuple of walls $\cup W_i$ such that each $W_{i+1}$ lies immediately above $W_i$. We call each wall $W_i$ a \emph{persistent wall} and $(W_i)$ a \emph{persistent tuple}. See \Cref{fig:pers-ann}.

\begin{figure}[ht]
    \centering
    \includegraphics[]{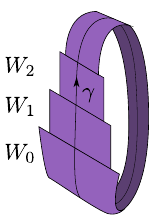}
    \caption{A persistent tuple of walls.}
    \label{fig:pers-ann}
\end{figure}

Restrict further to when $\gamma$ meets the walls of $\mathcal{B}$, but not the vertical one-skeleton. This means that $\gamma$ is non-singular, and $\gamma$ is contained in the union of a persistent tuple $(W_i)$. The intersection $x = \gamma\cap W_i$ is a one-cell. We can push $x$ into the box on either side of $W_i$ to obtain a vertical segment disjoint from the walls. The push-off moves each endpoint of $x$ from the boundary of a rectangle to its interior. We take both choices of push-off, and repeat for all the $W_i$. Since the pushed-off segments for $W_i$ and $W_{i+1}$ naturally glue top to bottom, all of the segments stitch together to give either one or two curves $\hat{\gamma}_i$. We call these the \emph{push-offs} of $\gamma$. Since each $\hat{\gamma}_i$ is disjoint from the walls, they have an itinerary $\hat{w}_i$ as above. These are the itineraries of $\gamma$.

For the remaining case, suppose that $\gamma$ lies in the vertical one-skeleton of $\mathcal{B}$. Now $\gamma$ is a union of vertical one-cells. Such a one-cell $x$ meets the boundary of multiple persistent walls. We can then push-off $x$ into any of these adjacent walls. We do so for each such wall, and repeat for all choices of $x$. The resulting vertical segments stitch together to give a collection of loops $\{\hat{\gamma}_i\}$. The precise number of these loops depends on the prong-type and rotation period of $\gamma$. Each $\hat{\gamma}_i$ is contained in the union of a persistent tuple, so we can apply the previous case to obtain a set of itineraries. Taking the union of these itineraries over all $\hat{\gamma}_i$ gives the itineraries of $\gamma$.
\end{con}

\begin{obs}\label{obs:twisted}
    Let $\gamma$ be a periodic orbit with itinerary $w$. The dual loop $\gamma_w$ is homotopic to some multiple of $\gamma$ by a homotopy that keeps the loops dual to the top and bottom rectangles of boxes. More precisely, suppose that $\gamma$ either meets the walls or the vertical one-skeleton and has a prong neighbourhood (\Cref{def:singular-orbit}) of type $N_{k,n}$. For $n\neq 0$, then for $t = k/\gcd(k,n)$, the loop $\gamma_w$ is homotopic to the multiple $\gamma^t$.
    
    Now suppose that there is a periodic orbit $\delta$ such that $\delta$ and $\gamma$ are corners of a common lozenge in $\mathcal{O}$. So, $\gamma^i\simeq \delta^{-j}$ where $i,j>0$ are the rotation periods of orbits \Cref{rem:precise-homotopies}. The preceding paragraph shows that the homotopy classes of $\gamma^i$ and $\delta^j$ are precisely the multiple encoded by their itineraries. 
\end{obs}
 
Choose an ordering on the top/bottom rectangles of the boxes. This endows itineraries with a lexical order. Let $\mathcal{I}_n(\mathcal{B})$ be the set of itineraries with length at most $n$ that are lexically maximal among their own cyclic permutations.

\begin{lemma}\label{lem:itins-not-shared}
	Two distinct periodic orbits $\alpha$ and $\beta$ of $\varphi$ have distinct itineraries in $\mathcal{I}_n(\mathcal{B})$.
\end{lemma}
\begin{proof}
	Let $w\in \mathcal{I}_n(\mathcal{B})$ be a shared itinerary of $\alpha$ and $\beta$. For $r$ a subrectangle on the bottom of box $B$, let $f(r)$ be the image of $r$ in the top of $B$ obtained by flowing vertically upwards. Let $(r_i)$ be the infinite sequence of rectangles passed through by following $w$ periodically forward. Set $s_0= r_0$ and $s_{i+1} = f(s_i)\cap r_{i+1}$. Then $\alpha$ and $\beta$ intersect $s_{i}$, while the Markov property forces $s_i$ to contract to a stable segment as $i\to\infty$. This implies $\alpha = \beta$.
\end{proof}

We want to scan $\mathcal{I}_n(\mathcal{B})$ for free homotopies between orbits of the flow. Duplicate itineraries of the same orbit give `false positives' of freely homotopic orbits, and so we need to detect and discard them. For this, we use the following.
\begin{obs}\label{obs:pers-wall}
    By construction, duplicate itineraries only occur in the cases where an orbit $\gamma$ that meets the walls or lies in the vertical one-skeleton. 
    
    Starting just with $\mathcal{B}$, if $(W_i)\subset\mathcal{B}$ is a persistent tuple, then there is a unique periodic orbit contained in the union of the $W_i$. The itineraries of this orbit are those of its push-offs as in \Cref{con:itin}.
    
    For $(W_i),(X_j)$ two persistent tuples, define a relation by setting $(W_i)\sim (X_j)$ if the orbits covered by $(W_i)$ and $(X_j)$ share an itinerary. The \emph{push-off equivalence} is the transitive closure of $\sim$. Because of \Cref{lem:itins-not-shared}, this identifies all duplicates: two itineraries in $\mathcal{I}_n(\mathcal{B})$ represent the same orbit if and only if they appear as itineraries for equivalent persistent tuples.
\end{obs}

There is a simple algorithm to compute the set of persistent walls of $\mathcal{B}$ in linear time. By following the persistent walls upward, we can arrange them into persistent tuples, and then sort them further into equivalence classes under the push-off relation.

Let $S$ be an equivalence class of persistent tuples of walls. Each tuple in $S$ has two itineraries; we take the union over $S$ of all these itineraries and discard all but one from $\mathcal{I}_n(\mathcal{B})$. Call the resulting set $\overline{\mathcal{I}}_n(\mathcal{B})$.
\begin{lemma}\label{lem:orbit-code}
    Let $\gamma\subset M$ be a periodic orbit of $\varphi$ that has an itinerary of length at most $n$. Then there is a unique itinerary of $\gamma$ in $\overline{\mathcal{I}}_n(\mathcal{B})$. \qed
\end{lemma}
    
\subsection{Finding a free homotopy}\label{subsec:findfit}
Let $\mathcal{B}^*$ denote the cell decomposition of $M$ that is dual to $\mathcal{B}$. Let $G$ be a presentation of $\pi_1(M)$ computed from the cell structure of $\mathcal{B}^*$. An element of $\overline{\mathcal{I}}_n$ gives a word in $G$. We use $\texttt{AreConjugate}(G,(w,v))$ to denote the solution to the conjugacy problem (\cite[Theorem 4.2]{sela}, Pr\'eaux \cite[Main Theorem]{preaux}) applied to the presentation $G$ and words $(w,v)$. We apply this solution to scan for pairs of itineraries in $\overline{\mathcal{I}}_n$ whose encoded orbits are the corners of a common lozenge in $\mathcal{O}$. This gives $\mathtt{FindFit}(n,\mathcal{B})$, the following simple algorithm.

\begin{algorithm}[ht]
		 \caption{ $\mathtt{FindFit}(n,\mathcal{B})$}
    \label{rou:find_fit}
		\begin{algorithmic}[1]
			\ForAll{$w,v \in \overline{\mathcal{I}}_n$}
                \If{\texttt{AreConjugate}$(G,(w,v^{-1}))$ = \texttt{True}}
                \State \Return $\mathtt{True}$
                \EndIf
                \EndFor
                \State \Return $\mathtt{False}$
		\end{algorithmic}
    \end{algorithm}

\begin{prop}\label{lem:find_fits_works}
    The $\mathtt{FindFit}(n,\mathcal{B})$ routine returns $\mathtt{True}$ for some $n>0$ if and only if $\varphi$ has perfect fits.
\end{prop}
\begin{proof}
    We use Fenley's characterisation, \Cref{thm:pfits-iff-homotopy}. If $\varphi$ has perfect fits, there is a pair of (possibly the same orbit)  that arises as the corners of a common lozenge in $\mathcal{O}$. The precise multiples that can occur (\Cref{rem:precise-homotopies}) are precisely those encoded by the itineraries by \Cref{obs:twisted}. A pair or single orbit arises in $\overline{\mathcal{I}}_n(\mathcal{B})$ for some $n$ by \Cref{lem:orbit-code}, at which point $\mathtt{FindFit}(n,\mathcal{B})$ returns $\mathtt{True}$. Since the itineraries in $\overline{\mathcal{I}}_n(\mathcal{B})$ are unique, the converse follows similarly.
\end{proof}
Note that since we are only looking for common corners of a lozenge, we may miss pairs of freely homotopic orbits that arise for smaller $n$.

\begin{rem}\label{rem:anosov-case}
We adapt \Cref{rou:find_fit} to the case where we are instead given a box decomposition of a (pseudo-)Anosov flow with marked orbits $(\varphi,\sing)$. Perform Dehn filling on the drilled manifold $M^{\circ}$ with slopes $s$ to obtain a flow $\varphi(s)$ that has singular orbits at each of the filled cusps (\Cref{con:gf}). We compute these slopes in terms of the boundary cell-structure induced by $\mathcal{B}^{\circ}$.

The periodic orbits of $\varphi(s)$ are precisely the images of the undrilled periodic orbits of $\varphi$ along with the cores of the filled solid tori. We then scan these orbits for conjugacies to see if $\varphi(s)$ has perfect fits. This is equivalent to $(\varphi,\sing)$ having no genuine perfect fits  by \Cref{lem:pfits-iff-homotopy-general}.
\end{rem}

\section{Finding no perfect fits}\label{sec:find_veering}
 Let $(\varphi,\sing)$ be a (pseudo-)Anosov flow with marked orbits (\Cref{def:marked-flow}). Let $\mathcal{B}$ be a box decomposition of $(\varphi,\sing)$. In this section, we give a routine to verify that $(\varphi,\sing)$ has no genuine perfect fits, \Cref{rou:find_fit}. The example to keep in mind is $\varphi$ a pseudo-Anosov flow and $\sing$ the set of singular orbits.

\subsection{The Agol-Gu\'eritaud construction}\label{subsec:agol-gueritaud}
We recall the conventions used in the Agol-Gu\'eritaud construction \cite{ag-talk}, following the expositions \cite[\S 4]{landry-minsky-taylor} and \cite[\S 2-5]{ss-loom}.
\begin{figure}[ht]
    \centering
    \includegraphics[width=0.6\textwidth]{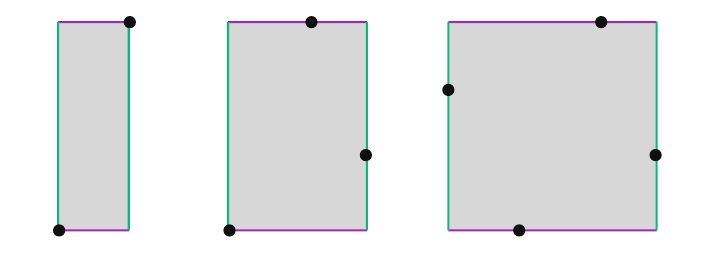}
    \caption{From left to right: An edge, face, tetrahedron rectangle.}
    \label{fig:skeletal_rects}
\end{figure}

\begin{defn}\label{defn:skel_rects}
    A \emph{cusp leaf} of $\mathcal{P}$ is a lift of a punctured leaf in $\mathcal{O}^{\circ}$. Let $r\subset \mathcal{P}$ be a rectangle with corresponding homeomorphism $f_r:r\to (0,1)^2$. Suppose that the boundary of $r$ is a union of cusp leaves. If:
    \begin{itemize}
        \item the homeomorphism $f_r$ extends to a homeomorphism
        \[ \bar{f}_r:r\to [0,1]^2- \{(0,0),(1,1)\}, \]
        we say that $r$ is an \emph{edge rectangle}.
        \item for some $x,y\in (0,1)$, the homeomorphism $f_r$ extends to a homeomorphism
        \[ \bar{f}_r:r\to [0,1]^2- \{(0,0),(x,1),(1,y)\}, \]
        we say that $r$ is a \emph{face rectangle}.
        \item for some $n,e,s,w \in (0,1)$ the homeomorphism $f_r$ extends to a homeomorphism
        \[ \bar{f}_r:r\to [0,1]^2- \{(0,w),(1,e),(s,0),(n,1)\}, \]
        we say that $r$ is a \emph{tetrahedron rectangle}.
    \end{itemize}
    See \Cref{fig:skeletal_rects}. In each case, we call $r$ a \emph{skeletal rectangle}.
\end{defn}
We emphasise the condition in \Cref{defn:skel_rects} that the boundary leaves are cusp leaves. This distinguishes an edge rectangle from the lift of a lozenge in $\mathcal{O}$ that is disjoint from $\sing$, since the two have the same foliation structure.

    Let $r\subset\mathcal{P}$ be a face or tetrahedron rectangle. Let $s\subset r$ be a subrectangle. If every unstable (stable) leaf that meets $r$ also meets $s$, we say that $s$ \emph{stable (unstable) spans} $r$. See \Cref{fig:span-n-orient}.

    By subdividing rectangles, we see:
    \begin{itemize}
        \item Each face rectangle contains three edge subrectangles.
        \item Each tetrahedron rectangles contains four face subrectangles and six edge subrectangles.
    \end{itemize}
   
    Given a face or tetrahedron rectangle $r\subset\mathcal{P}$, there is a unique edge subrectangle that stable spans and another which unstable spans. We denote these rectangles as $\sspan(r)$ and $\uspan(r)$, respectively.

We fix orientations on the transverse foliations $\mathcal{P}^s$ and $\mathcal{P}^u$, which we can do since $\mathcal{P}$ is simply connected. Following \cite[Remark 2.5]{ss-loom}, we consider the south-north direction as pointing forward in the stable direction, and the west-east as forward in the unstable direction See \Cref{fig:span-n-orient}. We refer to the north, south, east, and west sides of a skeletal rectangle, and the corners of a rectangle by their intercardinal positions. In terms of these conventions, an edge rectangle has a cusp at either its north-west or north-east corner, in which case we assign the \emph{veer}, denoted $\mathrm{Veer}(e)$, as blue or red respectively. See \Cref{fig:span-n-orient}. The veer is preserved under the $\pi_1(M^{\circ})$-action.
 
\begin{figure}[ht]
    \centering
    \includegraphics[scale=0.9]{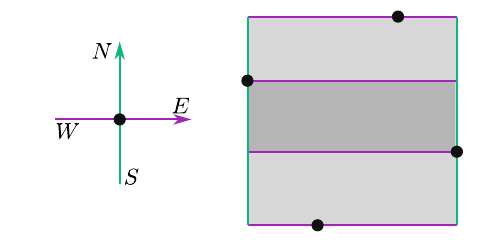}
    \caption{The shaded edge rectangle spans the tetrahedron rectangle in the unstable direction. The shaded edge rectangle is blue.}
    \label{fig:span-n-orient}
\end{figure}

Now we describe a finite form of the Agol-Gu\'eritaud construction.
\begin{con}[Agol-Gu\'eritaud]\label{con:ag}
Let $\mathcal{R}$ be a finite collection of tetrahedron rectangles in $\mathcal{P}$. Replace $\mathcal{R}$ by a maximal subset that contains no pair of tetrahedron rectangles that are $\pi_1(M^{\circ})$-translates.

To each rectangle $r\in \mathcal{R}$, assign an ideal tetrahedron $t(r)$. Label each vertex of $t(r)$ by a unique cusp of $r$. Let $e_i$ be the edges of $t(r)$ for $i\in\{1,\dots,6\}$. If $e_i$ has vertices labelled by cusps $\alpha,\beta$, we associate to $e_i$ the edge subrectangle $e_i\subset r$ that has corners $\alpha$ and $\beta$. Similarly, the faces $f_i$ correspond to a unique face rectangle $f_i\subset r$ for $i$. Endow $t(r)$ with the following veering structure:
\begin{itemize}
        \item Angles: If $e_i$ spans $r$ in either direction, assign $e_i$ a dihedral angle of $\pi$. Otherwise, assign $e_i$ a dihedral angle of $0$.

        \item Co-orientations: If $\sspan(f_i) = \sspan(r)$, assign $f_i$ an outward co-orientation, otherwise (where we must have $\uspan(f) = \uspan(r)$) assign an inward co-orientation.

        \item Colours: Give the edge $e_i$ the colour $\mathrm{Veer}(e_i)$.

    \end{itemize}
    The tetrahedron $t(r)$ can be depicted with each edge, face, and tetrahedron lying in the interior of its associated skeletal rectangle. See \Cref{fig:ag-tet}.
    
    \begin{figure}[ht]
             \centering
             \includegraphics[scale=0.75]{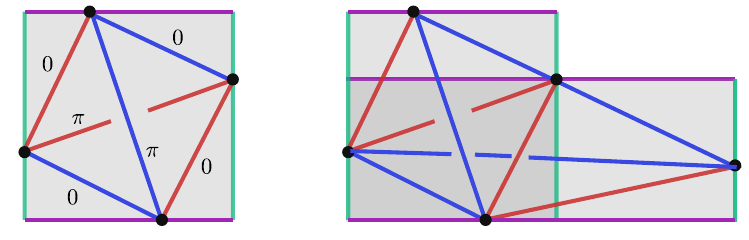}
             \caption{Left: a veering tetrahedron and its rectangle. Right: an induced face pairing between two veering tetrahedra.}
             \label{fig:ag-tet}
         \end{figure}
         
    Now we assign face pairings between the tetrahedra $\{t(r)\}_{r\in\mathcal{R}}$. Let $f, f'$ be a pair of face rectangles that are subrectangles of rectangles $r,\,r'\in 
   \mathcal{R}$. If $f$ and $f'$ are $\pi_1(M^{\circ})$-translates, assign a pairing between the corresponding faces of $F$ and $F'$. This pairing corresponds to extending $f$ into $r$ one way and into (some translate of) $r'$ the other, see the right of \Cref{fig:ag-tet}.
    
    Let $\mathcal{V}(\mathcal{R})$ be the tetrahedra $\{t(r)\}_{r\in\mathcal{R}}$ with the declared face pairings.
\end{con}

We emphasise that \Cref{con:ag} is stated so that it can be applied to a finite collection $\mathcal{R}$ regardless of whether $(\varphi,\sing)$ has genuine perfect fits.

\subsection{Characterising no perfect fits}
We use \Cref{con:ag} to characterise the absence of perfect fits.

\begin{defn}\label{def:complete}
    Let $\mathcal{R}$ be a finite collection of tetrahedron rectangles in $\mathcal{P}$. Call $\mathcal{R}$ \emph{complete} if $\mathcal{V}(\mathcal{R})$ is a veering triangulation.
\end{defn}

The purpose of this subsection is to establish the following.
\begin{prop}\label{lem:no-fits-iff-complete}
    The flow with marked orbits $(\varphi,\sing)$ has no genuine perfect fits if and only if $\mathcal{P}$ admits a complete collection of tetrahedron rectangles.
\end{prop}

The key property of a complete collection $\mathcal{R}$ is that all faces of $\mathcal{V}(\mathcal{R})$ are paired. However, it simplifies the exposition to just ask for the entire triangulation. Note that the property of $\mathcal{V}(\mathcal{R})$ being a veering triangulation is checkable.

\Cref{lem:no-fits-iff-complete} is essentially known in the literature, though in different forms. Consider first the forward direction. If $\floworbs$ has no perfect fits, then $\mathcal{P}$ is \emph{loom} in the sense of \cite[Definition 2.11]{ss-loom}. Then we apply the original Agol-Gu\'eritaud construction following \cite[\S 5]{ss-loom} to obtain a locally finite veering triangulation. Since tetrahedron rectangles are finite up to the $\pi_1(M)$-action by \cite[Lemma 4.5]{landry-minsky-taylor}, there is a finite complete collection $\mathcal{R}$. The fact that $\mathcal{P}$ is a loom space for a flow without genuine perfect fits is (implicitly) asserted in \cite[\S 4]{landry-minsky-taylor}, and a proof will appear in future work of Schleimer-Segerman's program (promised in \cite[Example 2.19]{ss-loom}).

So, we establish the backward direction of \Cref{lem:no-fits-iff-complete}. We first prove a sequence of lemmas about tetrahedron rectangles.

Let $\mathcal{R}$ be a complete collection of tetrahedron rectangles. Let \[\overline{\mathcal{R}} = \{g\cdot R: g\in\pi_1(M^{\circ})\}\] be the collection of all translates of the rectangles in $\mathcal{R}$.  We first show that $\overline{\mathcal{R}}$ covers $\mathcal{P}$. Since every edge rectangle spans two tetrahedron rectangles up to translation, every face rectangle in $\overline{\mathcal{R}}$ is contained in two tetrahedron rectangles.

We call a sequence of tetrahedron rectangles $(r_n)_{i\in\bbN}\subset \overline{\mathcal{R}}$ \emph{layered} if:
\begin{itemize}
    \item consecutive rectangles $r_{i}$ and $r_{i+1}$ share a face subrectangle $f_i$
    \item the face $f_i$ spans $r_i$ in a common (stable/unstable)  direction for all $i$.
\end{itemize}
 See \Cref{fig:layered-sequence}. A layered sequence corresponds to layering the tetrahedron associated to $r_{i+1}$ in $\mathcal{V}(\mathcal{R})$ over the tetrahedron associated to $r_i$ in the universal cover. The spanning condition enforces that the layering is always up or always down.
\begin{figure}[ht]
    \centering
    \includegraphics{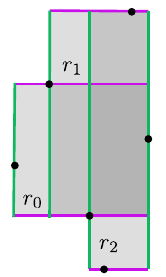}
    \caption{Rectangles in a stable layered sequence.}
    \label{fig:layered-sequence}
\end{figure}

\begin{lemma}\label{lem:bad-layered}
     Let $(r_i)\subset \overline{\mathcal{R}}$ be an infinite, layered sequence of tetrahedron rectangles. Then the rectangles $r_i$ cannot all share a common edge rectangle.
\end{lemma}
\begin{proof}
Suppose for contradiction that the sequence $(r_i)$ contains a common edge rectangle. The tetrahedra corresponding to the rectangles $r_i$ all share a common edge, as in \Cref{fig:bad-layered}. So, the corresponding edge has infinite degree on the universal cover. This contradicts that $\mathcal{V}(\mathcal{R})$ is a veering triangulation.
\end{proof}

\begin{con}\label{con:leaf-cover}
    Let $l\subset \mathcal{P}$ be a half of (that is, a ray contained in) a leaf that meets $r\in \overline{\mathcal{R}}$. We cover $l$ with a layered sequence of tetrahedron rectangles $(r_n)$, as follows.
    
    Breaking symmetry, suppose that $l$ is half of a stable leaf extending southward. The finite end of $l$ is contained in some $r_0\in\overline{\mathcal{R}}$. There are two vertically spanning face subrectangles of $r_0$. At least one of these face rectangles $f_0\subset r_0$ contains $l\cap r_0$, with all three possible configurations shown in \Cref{fig:leaf-spots}.
\begin{figure}[ht]
    \centering
    \includegraphics{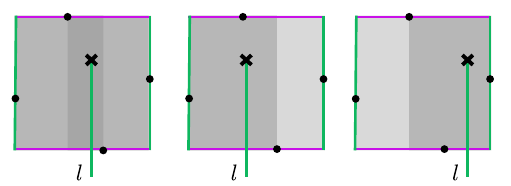}
    \caption{The shaded faces are those which can be extended along to give a larger neighbourhood of $l$.}
    \label{fig:leaf-spots}
\end{figure}

If there is a choice for $f_0$, then one of the choices can be extended over its south edge to a tetrahedron rectangle. Take this as $f_0$. With $f_0$ chosen, let $r_1$ be the other tetrahedron rectangle (not $r_0$) that $f_0$ spans. We again have a choice of unstable-spanning face rectangle $f_1$ of $r_1$, and proceed inductively, always preferring the face rectangle $f_i$ which extends southward when given a choice. This gives a layered sequence of tetrahedron rectangles $(r_i)$. Note that the rectangles $r_i$ are strictly decreasing in the unstable direction and increasing in the stable direction.

If $l$ is the end of a cusp leaf which meets a cusp, then the sequence $(r_i)$ can terminate with a rectangle that has this cusp on an edge. Otherwise, the sequence is infinite.
\end{con}

\begin{lemma}\label{lem:layered-covers}
    There is a layered sequence $(r_n)$ that covers $l$.
\end{lemma}
\begin{proof}
    Breaking symmetry, take $l$ to be the south half of a stable leaf. Since the sequence $(r_n)$ decreases in the unstable direction, it suffices to show that there is a subsequence of $(r_n)$ that is obtained by extending only over the south edges of the preceding rectangle. If this were not the case, the tail of $(r_n)$ is obtained by only extending northward. This occurs only when $l$ meets each rectangle of the tail as in the rightmost of \Cref{fig:leaf-spots}. This implies that the rectangles must share a common edge subrectangle $e\subset\mathcal{P}$ in their south-east corners.
    
    If the layered sequence terminates after finitely many rectangles, then $l$ must be the side of cusp leaf meeting a cusp. In such a case, the rectangles then cover $l$.
    
    So, suppose that $l$ does not end in a cusp. Then the sequence of rectangles $(r_n)$ does not terminate. This contradicts \Cref{lem:bad-layered}.
\end{proof}
\begin{figure}[ht]
    \centering
    \includegraphics[]{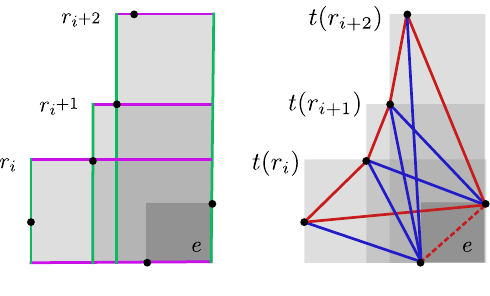}
    \caption{A layered sequence (left) which corresponds to layering tetrahedra of $\mathcal{V}(\mathcal{R})$ (right) around a fixed edge $e$ (dotted).}
    \label{fig:bad-layered}
\end{figure}
\begin{lemma}\label{lem:complete-implies-covers}
    Let $\mathcal{R}$ be a complete collection of tetrahedron rectangles. Then $\overline{\mathcal{R}}$ covers $\mathcal{P}$.
\end{lemma}
\begin{proof}
By \Cref{lem:layered-covers}, $\overline{\mathcal{R}}$ covers every leaf of $\mathcal{P}$. Since every point lies on a leaf, $\overline{\mathcal{R}}$ covers $\mathcal{P}$.
\end{proof}

\begin{lemma}\label{lem:complete-implies-no-pfits}
    Suppose that $\mathcal{P}$ admits a complete collection of tetrahedron rectangles $\mathcal{R}$. Then $(\varphi,\sing)$ has no genuine perfect fits.
\end{lemma}
\begin{proof}
    Suppose for contradiction that $(\varphi,\sing)$ has perfect fits. Then there is a genuine perfect fit rectangle $Q\subset\mathcal{P}$ (\Cref{def:rel_fits}).  Breaking symmetry, assume that the perfect fit occurs at the north-east corner of $Q$. Let $p\in Q$ be the south-east corner of $Q$. Since $\overline{\mathcal{R}}$ covers $\mathcal{P}$ (\Cref{lem:complete-implies-covers}), there exists a rectangle $r_0\in \overline{\mathcal{R}}$ with $p\in r_0$.

    Let $l\in\mathcal{P}^s$ be the stable leaf through $p$; this leaf forms the perfect fit in $Q$. Using \Cref{con:leaf-cover}, we build a layered sequence of rectangles $(r_n(l))$ that covers a half of $l$ forming the perfect fit, using $r_0(l) = r_0$. This sequence does not terminate, since the half of $l$ we are covering does not contain a cusp. Let $k\in\mathcal{P}^u$ be the unstable leaf through $p$, on the south side of $Q$. The leaves $l$ and $k$ cut $r_n$ into \emph{quarters}, and we refer to these quarters using the intercardinal directions.
    
    Since the sequence $r_i$ decreases in the unstable direction, there is $N$ such that the western edge of $r_N$ meets the interior of $Q$. Since $Q$ is a perfect fit rectangle, there can be no cusps where $r_N$ meets $Q$. So, there are no cusps on the north-west quarter of $r_N$, as in the left of \Cref{fig:tet-loz-sequence}.
    
    Since $(r_i)$ covers $l$, there is $N$ such that $r_{N+1}$ is obtained by extending over the north edge of $r_N$. If there are no cusps in the south-east quarter of $r_N$, choose $s_0=r_N$. If there are cusps in the south-east quarter of $r_N$, we can perform one or two face extensions to $r_N$ (\Cref{fig:tet-loz-sequence}, left) to obtain a rectangle $s_0\in\overline{\mathcal{R}}$ such that $s_0$ meets $l$ and $s_0$ does not have cusps in its south-east quarter. Now we build a new layered sequence starting from $s_0$ that covers the same end of $l$ as $r_n$. With this new sequence $(s_n)$, we must always extend over the north edge (see \Cref{fig:tet-loz-sequence}, right). This gives a layered sequence that meets a fixed common south-west edge rectangle, which contradicts \Cref{lem:bad-layered}.
\end{proof}
\begin{figure}[ht]
    \centering
    \includegraphics{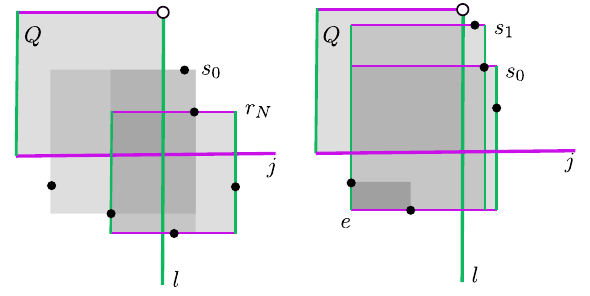}
    \caption{Sequences of tetrahedron rectangles intersecting a lozenge as in the proof of \Cref{lem:complete-implies-no-pfits}}
    \label{fig:tet-loz-sequence}
\end{figure}

\Cref{lem:complete-implies-no-pfits} completes the proof of \Cref{lem:no-fits-iff-complete}. \qed

\begin{rem}
The pseudo-Anosov flows constructed in the program of Schleimer-Segerman, which are built as an inverse to the original Agol-Gu\'eritaud construction, do not have perfect fits. This is a consequence of \cite[Theorem 8.1]{ss-link} and \cite[Lemma 9.2]{ss-link}. \Cref{lem:complete-implies-no-pfits} is an analogue of this fact for the finite version, \Cref{con:ag}.
\end{rem}

The main routine of the section, $\mathtt{FindVeering}$, will apply \Cref{con:ag} and check the criterion \Cref{lem:no-fits-iff-complete}. The remaining work is to build, from $\mathcal{B}$, the required skeletal rectangles. This comprises the remainder of the section.

\subsection{Boxes in the universal cover}
We now use $\mathcal{B}_M$ to refer to the box decomposition of the flow $\varphi$ with marked orbits $\sing$. Let $\mathcal{B}_{M}^{\circ}$ be the resulting (drilled) box decomposition of $M^{\circ}$. Lift $\mathcal{B}_{M}^{\circ}$ to a box decomposition $\mathcal{B}_{N}$ of the universal cover $N$ of $M^{\circ}$. We work primarily with $\mathcal{B}_{N}$ for the remainder of this section and so denote it by $\mathcal{B}$.
     
  We introduce some nomenclature for working with $\mathcal{B}$. Declare two cusp rectangles in $\mathcal{P}$ as equivalent if their projections to $\mathcal{O}$ share a corner in $\sing$. A \emph{cusp} of $\mathcal{P}$ is an equivalence class of cusp rectangles (\cite[Definition 3.3]{ss-loom}). Cusps in $\mathcal{P}$ lift naturally to \emph{cusps} of the lifted flow space $N$. For $\alpha$ a cusp of $N$, we have the \emph{stable/unstable cusp leaves} of $N$ as the lifts of the cusp leaves in $\mathcal{P}$. We use $N^s(\alpha)$ and $N^u(\alpha)$ to denote the set of all stable/unstable cusp leaves through $\alpha$. We use $N^s$ and $N^u$ to denote set of stable and unstable leaves of $N$, respectively.
  
  In the other direction, we call the quotient of some subset $X\subset N$ in $\mathcal{P}$ the \emph{shadow} of $X$.

       If $B\in\mathcal{B}$, we use $P(B)\in \mathcal{B}_M$ to denote the box that we drill and lift to obtain $B$. If $P(B)$ meets $\sing$, then $B$ meets cusps of $N$, and we say that (some of) the cusp leaves of such a cusp \emph{emanate into $B$}. When this cusp leaf contains a wall of $B$, we say it emanates into this wall, and otherwise that it emanates into the interior of $B$.
        
     Call a gluing between $A$ and $B$ \emph{vertical} if it occurs between the tops and bottoms of boxes, otherwise \emph{horizontal}.
     
    We define a partial order on $\mathcal{B}$. For $A, B\in\mathcal{B}$, declare $A\leq B$ if there is a chain of boxes $A=A_0,\dots,A_n= B$ such that the top of $A_i$ is glued to the bottom of $A_{i+1}$. We say that $A$ lies \emph{above} $B$ and that $B$ lies \emph{below} $A$. Vertically adjacent boxes and their shadows are dictated by the Markov property (\Cref{def:item:markov} of \Cref{def:pa-flow}), see \Cref{fig:vert-glues}. The shadow of horizontally adjacent boxes $A, B\in\boxes$ is a pair of rectangles with disjoint interiors, as in \Cref{fig:horiz-glues}.
        \begin{figure}[ht]
            \centering
                 \begin{subfigure}[b]{0.45\textwidth}
                    \centering
                    \includegraphics[]{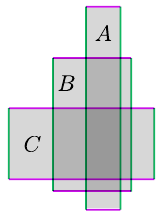}
                    \caption{$A$ lies above $B$ lies above $C$}
                    \label{fig:vert-glues}
                \end{subfigure}
                \begin{subfigure}[b]{0.45\textwidth}
                    \centering
                    \includegraphics[]{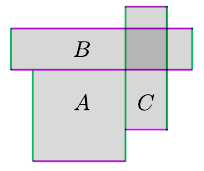}
                    \caption{$A$ is adjacent to $B$ and $C$ while $C$ lies above $B$}
                     \label{fig:horiz-glues}
                \end{subfigure}
        \caption{Shadows of adjacent boxes of $\mathcal{B}$ in $\mathcal{P}$.}
        \end{figure}
        
        An essential property of the gluings is the following.
        \begin{lemma}\label{lem:pers-walls-work}
             Let $\alpha$ be a cusp of $N$. Let $L\in N^s(\alpha)$ be a stable (unstable) cusp leaf that emanates from $\alpha$ and into a box $B\in\mathcal{B}$. Suppose that $p\in L$. Then there exists a box $A\in\mathcal{B}$ such that:
             \begin{itemize}
                 \item the box $A$ lies above (below) $B$
                 \item the leaf $L$ emanates into $A$
                 \item the shadow of $A$ contains that of $p$.
             
             \end{itemize}
        \end{lemma}
        \begin{proof}
             Let $\mathcal{A}$ be the collection of all boxes $A$ lying above $B$ such that $L$ emanates into $A$. As one walks upward through $\mathcal{A}$, the boxes must eventually stretch in the stable direction because of the Markov property (see \Cref{def:pa-flow}). So, the boxes must extend over $L$. See \Cref{fig:pers-wall}.
            \begin{figure}[ht]
            \centering
                \includegraphics[]{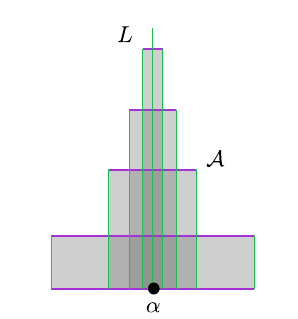}
                \caption{The boxes that meet a cusp leaf as in \Cref{lem:pers-walls-work}.}
                \label{fig:pers-wall}
            \end{figure}
        \end{proof}
        
        We work with $\mathcal{B}$ in practice by working with large finite subsets of boxes. Let $\mathcal{G}_M\subset M$ be the graph dual to the cell structure of the drilled box decomposition $\mathcal{B}_{M}^{\circ}$. The lift of $\mathcal{G}_M$ to $\mathcal{G}\subset N$ is a graph dual to the cell structure on $\mathcal{B}$. We stress that $\mathcal{G}$ includes gluings between the sides of the boxes and contains the underlying graph of $\mathcal{M}(\mathcal{B})$ as a proper subset.
        
        \begin{defn}\label{def:box-ball}
        Choose a point $x\in \mathcal{G}$ that descends into a chosen box $X\in\mathcal{B}$. Use $\mathcal{G}(x,n)\subset \mathcal{G}$ to denote the $n$-ball about $x$. Let $\mathcal{B}(X,n)$ be the collection of boxes in $\mathcal{B}$ that correspond to vertices in $\mathcal{G}(x,n)$. We call $\mathcal{B}(X,n)$ the \emph{$n$-ball of boxes}.
        \end{defn}
        
        The choices of $x$ and $X$ do not matter for the main algorithm, so we suppress them and talk of $\mathcal{G}(n)$ and $\mathcal{B}(n)$.
        
        Since $\mathcal{G}$ is a Cayley graph for $\pi_1(M)$, we appeal to the solution to the word problem for $\pi_1(M^{\circ})$ (\cite[Theorem 4.2]{sela} and Pr\'eaux \cite[Main Theorem]{preaux}) to generate $\mathcal{B}(n)$. Note that $\mathcal{B}(n)$ need not be simply connected.
    
        Since $N^s$ and $N^u$ are orientable, we can orient the foliations within the boxes in $\mathcal{B}(n)$ by choosing, in the initial box $X$, which stable walls are east and west, and which unstable walls are north and south (as per the orientation convention on $\mathcal{P}$ in \Cref{subsec:agol-gueritaud}). Every box in $\mathcal{B}(n)$ then inherits an orientation by induction.
        
\begin{rem}\label{rem:projecting-is-hard}
        Rather than working in the three-manifold $N$, one might be tempted to work directly in $\mathcal{P}$ and with shadows of boxes. Since the shadow of $\mathcal{B}$ is not locally finite, a starting point would be to work with the projection of $\mathcal{B}(n)$ to $\mathcal{P}$. However, there is no obvious way to determine if two boxes of $B, A\in\mathcal{B}(n)$ do not share an orbit. This would be possible with quantitative bounds, see \Cref{qn:quasigeodesic}.
      \end{rem}

\subsection{Finding rectangles}\label{subsec:find_rects}
This is a technical subsection in which we compute, in terms of boxes, representatives of the skeletal rectangles in $\mathcal{P}$.

Let $\mathcal{A}\subset\mathcal{B}$ be a finite connected subcomplex of boxes. To find rectangles in the shadow of $\mathcal{A}$, we use discrete versions of cusps and leaves. These should be thought of the cusps and leaves which $\mathcal{A}$ `can see'.
\begin{defn}\label{def:cuspsnleaves}
    A \emph{cusp} of $\mathcal{A}$ is a pair $(s,\{B_i\})\subset \mathcal{A}$ where $s\in\sing$ is a marked orbit of $\varphi$ and $\{B_i\}\subset\mathcal{A}$ is a maximal connected collection of drilled boxes such that each parent box $P(B_i)\in \mathcal{B}_M$ meets $s$. See \Cref{fig:cusps-in-boxes}. We use $\cusps(\mathcal{A})$ to denote the set of all cusps of $\mathcal{A}$.
\begin{figure}[ht]
     \centering
     \includegraphics[scale=0.25]{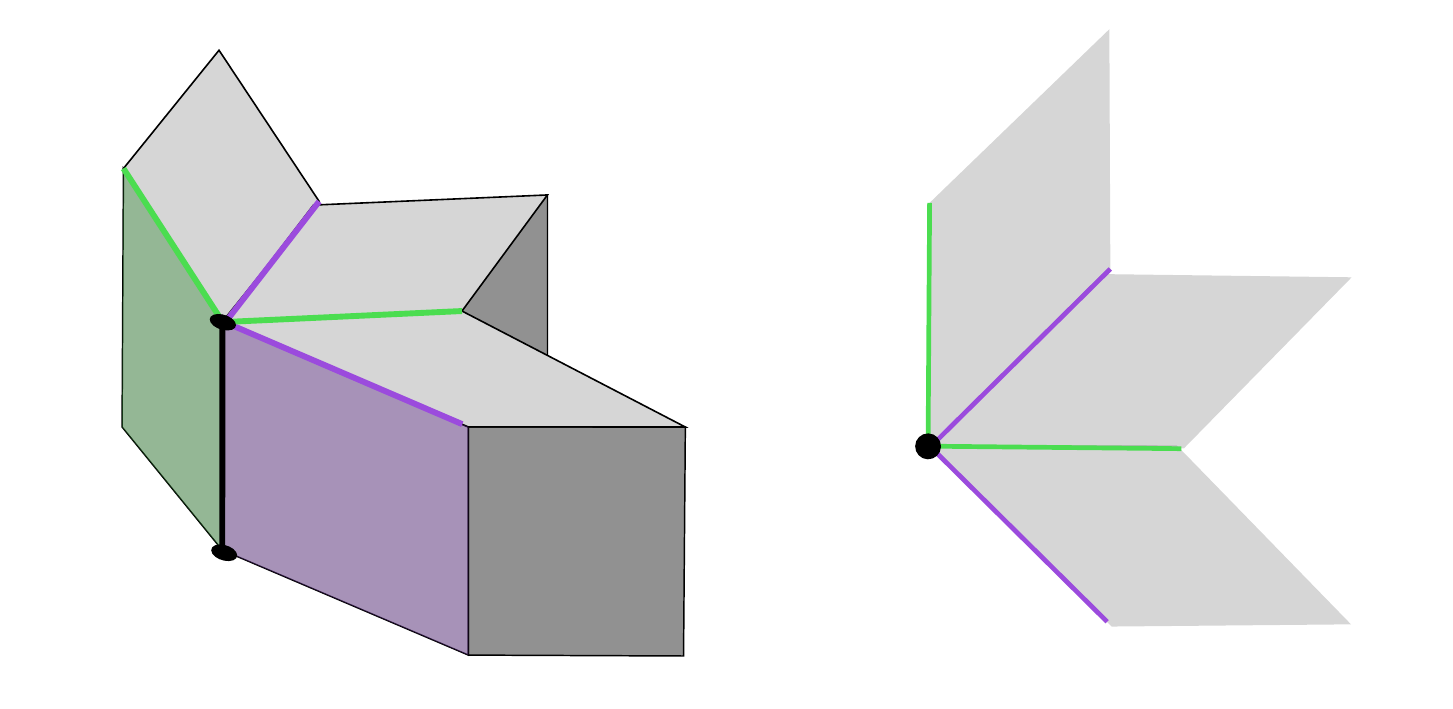}
     \caption{Left: Three boxes which define a common element of $\mathrm{Cusps}(\mathcal{A})$, and several cusp leaves emanating into these boxes. Right: The shadows in $\mathcal{P}$.}
     \label{fig:cusps-in-boxes}
 \end{figure}
 \end{defn}
 
\begin{rem}\label{rem:no_proj_dupes}
    A cusp of $\mathcal{A}$ corresponds to some underlying cusp $\alpha$ of $N$, and we often use $\alpha$ to denote the corresponding cusp of $\mathcal{A}$. However, distinct cusps of $\mathcal{A}$ may correspond to the same cusp of $N$. Verifying otherwise runs into the same difficulty as in \Cref{rem:projecting-is-hard}. However, it is trivial to check whether two corresponding cusps of $N$ are translates: it is when the marked orbit is the same.
\end{rem}
Define a partial order on $\mathcal{A}$ by declaring $B\leq A$ when there is an ascending vertical chain of boxes from $B$ to $A$ contained in $\mathcal{A}$ (this partial order is weaker than that inherited from $\mathcal{B}$). For a collection of boxes $\{B_i\}$ in $\mathcal{A}$, we use $\max_i(B_i)$ to denote the set of maximal boxes in $\{B_i\}$ with respect to this partial order on $\mathcal{A}$. Given a cusp $(\alpha,\{B_i\})$, a stable leaf $L$ meeting this cusp emanates into one or two boxes in $\max_i(B_i)$. For $B\in \max_i(B_i)$, the Markov property ensures that $L$ emanates into each box in a vertical ascending chain up to $B$. We use a single box to determine each cusp leaf by enforcing a tie-breaker.

\begin{defn}\label{def:leaves}
    Fix a cusp $\alpha = (s,\{B_i\})\in \cusps(\mathcal{A})$. Let $L\in N^s(\alpha)$. There is at most one box $B_L\in \mathrm{max}_i(B_i)$ such that $L$ emanates into the interior or into the west wall of $B_L$. If there is such a box, call the pair $(B_L,L)$ a \emph{stable cusp leaf of $\alpha$}. We call $B_L$ the \emph{base box} of the cusp leaf. Similarly, we define \emph{unstable cusp leaves}, now using minimal boxes, and with leaves emanating into the interior or into the south wall of the boxes.
\end{defn}

 Note that a fixed box can be the base box for two distinct stable (unstable) leaves when there is a stable (unstable) leaf emanating into the interior.
 
 We use $\mathcal{A}^s$ and $\mathcal{A}^u$ to denote the sets of all stable and unstable cusp leaves, respectively. Like cusps, we have potential duplicates, that is, distinct pairs $(B_L,L)$, $(B'_L,L)\in \mathcal{A}^s$ for the same underlying leaf $L$.
  
\begin{defn}\label{def:walls}
    Let $(B_L,L)\in \mathcal{A}^s$ be a stable (unstable) cusp leaf. Let $\mathcal{W}(\mathcal{A},(B_L,L))$ be the collection of all boxes in $\mathcal{A}$ that lie below $B_L$ in $\mathcal{A}$ (above $B_L$). We call $\mathcal{W}(\mathcal{A},(B_L,L))$ the \emph{wall of boxes} about $(B_L,L)$ in $\mathcal{A}$.
\end{defn}

\begin{figure}[ht]
    \centering
    \includegraphics[width=0.3\textwidth]{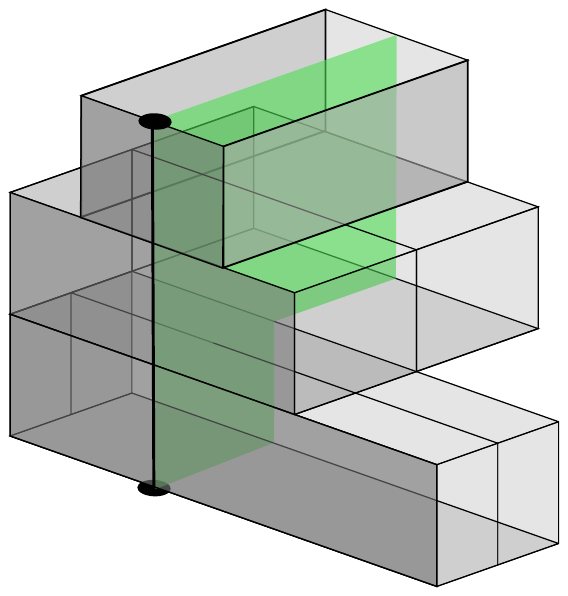}
    \caption{A wall of boxes.}
    \label{fig:wall-of-boxes}
\end{figure}

In $N$, the wall of boxes $\mathcal{W}(\mathcal{A},(B_L,L))$ is a combinatorial neighbourhood of a connected component of where $L$ meets $\mathcal{A}$, see \Cref{fig:wall-of-boxes}.

\begin{defn}
    Let $(L,K)\in N^s\times N^u$ be a pair of cusp leaves. Call $(L,K)$ a \emph{joint} if $\mathcal{W}(\mathcal{A},(B_L,L))\cap \mathcal{W}(\mathcal{A},(B_K,K))\neq\varnothing$ for some $(B_L,L)$,\, $(B_K,K)$. We use $\ints(\mathcal{A})$ to denote the set of joints.
\end{defn}
\begin{lemma}\label{lem:jointswork}
    If $(L,K)\in\ints(\mathcal{A})$, the pair of leaves $L$ and $K$ intersect in $N$.
\end{lemma}
\begin{proof}
    If $(L,K)\in\ints(\mathcal{A})$, there is $B\in \mathcal{W}(\mathcal{A},(B_L,L))\cap \mathcal{W}(\mathcal{A},(B_K,K))$. By the Markov property, $L$ and $K$ both meet $B$. So, $L$ and $K$ intersect in $B$.
\end{proof}
The boundary of a skeletal rectangle is made up of joints; see \Cref{fig:walls-intersecting}.

\begin{figure}[ht]
    \centering
    \includegraphics{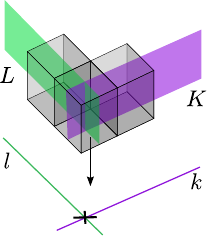}
    \caption{A pair of cusp leaves in $\ints(\mathcal{A})$, projecting to part of the boundary of a skeletal rectangle in $\mathcal{P}$.}
    \label{fig:walls-intersecting}
\end{figure}

\begin{lemma}\label{lem:walls-find-ints}
    Let $\alpha, \beta$ be cusps of $N$. Let $L\in N^s(\alpha)$ and $K \in N^u({\beta})$ be stable and unstable cusp leaves, respectively. Then $L$ and $K$ intersect if and only if for some $n>0$ we have $(L,K)\in\ints(\mathcal{B}(n))$.
\end{lemma}
\begin{proof}
The forward direction follows immediately from \Cref{lem:jointswork}. So, suppose that $L$ and $K$ intersect. Then they do so in some box $B\in\mathcal{B}$. By compactness of $M$, there is $n$ such that $B\in\mathcal{B}(n)$. By \Cref{lem:pers-walls-work}, there is $B_L\in\mathcal{B}$ lying above $B$ and  $B_K\in\mathcal{B}$ below $B$ such that $L$ and $K$ emanate into $B_L$ and $B_K$ respectively. This implies that for $n$ sufficiently large, we have $B,B_L,B_K\in\mathcal{B}(n)$. In turn, $B\in \mathcal{W}(\mathcal{B}(n),(B_L,L))\cap \mathcal{W}(\mathcal{B}(n),(B_K,K))$ and $(L,K)\in\ints(\mathcal{B}(n))$.
\end{proof}

Now we build skeletal rectangles. Suppose that $e\subset\mathcal{P}$ is an edge rectangle. Suppose that there are two joints in $\ints(\mathcal{A})$ whose projections to $\mathcal{P}$ bound $e$. We call these joints the \emph{joints of $e$} and say that $e$ \emph{is shaded by} $\mathcal{A}$. We let $\mathrm{Edges}(\mathcal{A})$ be the set of all edge rectangles $e$ that are shaded by $\mathcal{A}$. We define $\mathrm{Faces}(\mathcal{A})$ as the set of face rectangles whose edge subrectangles are shaded by $\mathcal{A}$ and similarly $\mathrm{Tets}(\mathcal{A})$ for tetrahedron rectangles.

\begin{con}\label{con:data}
We build the collections $\edges$, $\faces$, $\tets$ using the combinatorics of $\ints(\mathcal{A})$. Namely:
\begin{itemize}
    \item If two pairs in $\ints(\mathcal{A})$ share an underlying pair of cusps in $\mathcal{A}$, the underlying leaves of the joints bound a unique edge rectangle in $\edges(\mathcal{A})$.
    
    \item If there are three (six) edge rectangles in $\edges(\mathcal{A})$ whose cusps are a subset of three (four) fixed cusps, these are edge subrectangles of a unique face rectangle in $\faces(\mathcal{A})$ (tetrahedron rectangle in $\tets(\mathcal{A})$).
\end{itemize}

For $r$ a skeletal rectangle, let $\mathtt{GetJoints}(r)$ be the function that returns the set of joints in $\ints(\mathcal{A})$ corresponding to $r$ and its subrectangles. With this, we compute the following data, which we require for \Cref{con:ag}:
\begin{itemize}
    \item For $e\subset\edges(\mathcal{A})$, choose $(L,K)\in\ints(e)$. Since each box in $\mathcal{A}$ is oriented, $L$ emanates from either the east or west side of a box. Similarly, $K$ emanates into the north or south side. This pair of orientations determines the veer of $e$.

    \item Let $f\subset \faces(\mathcal{A})$ with cusps $\alpha,\,\beta,\, \delta\in\cusps(\mathcal{A})$. In the following, we use $L^{\alpha}$ to denote a cusp leaf whose underlying cusp is $\alpha$. Then the three edge subrectangles $e_1,\,e_2,\,e_3\subset f$ have the form:
    \begin{itemize}\setlength{\itemindent}{0.5in}
            \item $\mathtt{GetJoints}(e_1) = \{(L^{\alpha}, K^{\beta}) ,(L^{\beta}, K^{\alpha})\}$
            \item $\mathtt{GetJoints}(e_2)= \{(L^{\alpha}, K^{\delta}),( L^{\delta}, K^{\alpha})\}$
            \item $\mathtt{GetJoints}(e_3) = \{(L^{\beta}, K^{\delta}), (I^{\beta}, J^{\delta})\}$
            \end{itemize}
     See \Cref{fig:subdivide}, left. The rectangles $e_1$ and $e_2$ are $\mathrm{Span}^s(f)$ and $\mathrm{Span}^u(f)$, respectively.
    
\begin{figure}[ht]
      \begin{center}
        \includegraphics[]{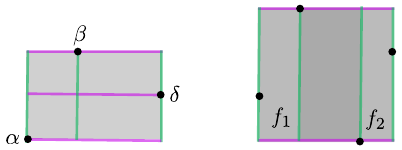}
      \end{center}
      \caption{Example configurations for \Cref{con:data}.}
      \label{fig:subdivide}
\end{figure}

    \item If $t\subset \tets(\mathcal{A})$, then there are two face subrectangles $f_1,f_2\subset t$ for which $\mathrm{Span}^s(f_1) = \mathrm{Span}^s(f_2)$. Then $\mathrm{Span}^s(t) =\mathrm{Span}^s(f_1)$. See \Cref{fig:subdivide}, right. We compute $\mathrm{Span}^u(t)$ similarly. \qedhere
\end{itemize}
\end{con}

\begin{lemma}\label{lem:pre-find-veering-works}
    The marked flow $(\varphi,\sing)$ has no genuine perfect fits if and only if there is $n$ such that $\mathrm{Tets}(\mathcal{B}(n))$ is complete.
\end{lemma}
\begin{proof}
    In the forward direction, if $\floworbs$ has no genuine perfect fits, there is a complete collection $\mathcal{R}$ by \Cref{lem:no-fits-iff-complete}. By \Cref{lem:walls-find-ints}, any intersection between cusp leaves of $N$ are eventually recorded in the form of joints in $\ints(\mathcal{B}(n))$. So, we can fix $n$ such that $\mathcal{R}\subset \mathrm{Tets}(\mathcal{B}(n))$, and $\mathrm{Tets}(\mathcal{B}(n))$ is complete.
    
    Conversely, suppose that $\floworbs$ has genuine perfect fits. By the converse of \Cref{lem:no-fits-iff-complete}, there is no complete collection in $\mathcal{P}$ at all. In particular, $\mathrm{Tets}(\mathcal{B}(n))$ is never complete.
\end{proof}

\subsection{Identifying translates}\label{subsec:translates}
The remaining data we need to apply \Cref{con:ag} is when the skeletal rectangles found in \Cref{subsec:find_rects} are translates. We compute this by building a canonical representative for a pair of intersecting leaves.

To ease exposition, we assume for this subsection that the $\pi_1(M^{\circ})$-action preserves the orientations on the foliations of $\mathcal{P}$. This ensures that the tie-breaker in the definition of cusp leaves \cref{def:leaves} is preserved under the action. In terms of finding perfect fits, this is inconsequential. This is because we can take a double cover of $M^{\circ}$ to have this property, and having a genuine perfect fit is a property on the universal cover. However, we will also want the precise veering triangulation for \cref{subsec:recog_flows}. We describe how to adapt the approach to the orientation reversing case in \cref{rem:orientation_fix}.

\begin{con}\label{def:cs}
    Let $(L,K)\in\ints(\mathcal{A})$. Thus, $\mathcal{W}(\mathcal{A},(B_L,L))\cap\mathcal{W}(\mathcal{A},(B_K,K))\neq\varnothing$ for some boxes $B_L \,,B_K$. We work in two cases. First, suppose that $L$ and $K$ do not emanate into a common box. Fix any box $B\in\mathcal{W}(\mathcal{A},(B_L,L))\cap\mathcal{W}(\mathcal{A},(B_K,K))$. Walk up from $B$ in $\mathcal{W}(\mathcal{A},(B_L,L))$ until you meet the lowest box $A_L$ such that $L$ emanates into the interior or the west wall of $A_L$. Similarly, walk down from $B$ in $\mathcal{W}(\mathcal{A},(B_K,K))$ until we reach a maximal box $A_K$ (with respect to the partial order) that $K$ emanates into. Let $\cs(L,K)$ denote the resulting chain of boxes. Call $\cs(L,K)$, paired with the data of which faces of $B_L$ and $B_K$ that $L$ and $K$ emanate into, the \emph{cornerstone} of $(L,K)$. See \Cref{fig:cornerstone}, above.
    
    \begin{figure}[ht]
        \centering
        \includegraphics[width=0.7\textwidth]{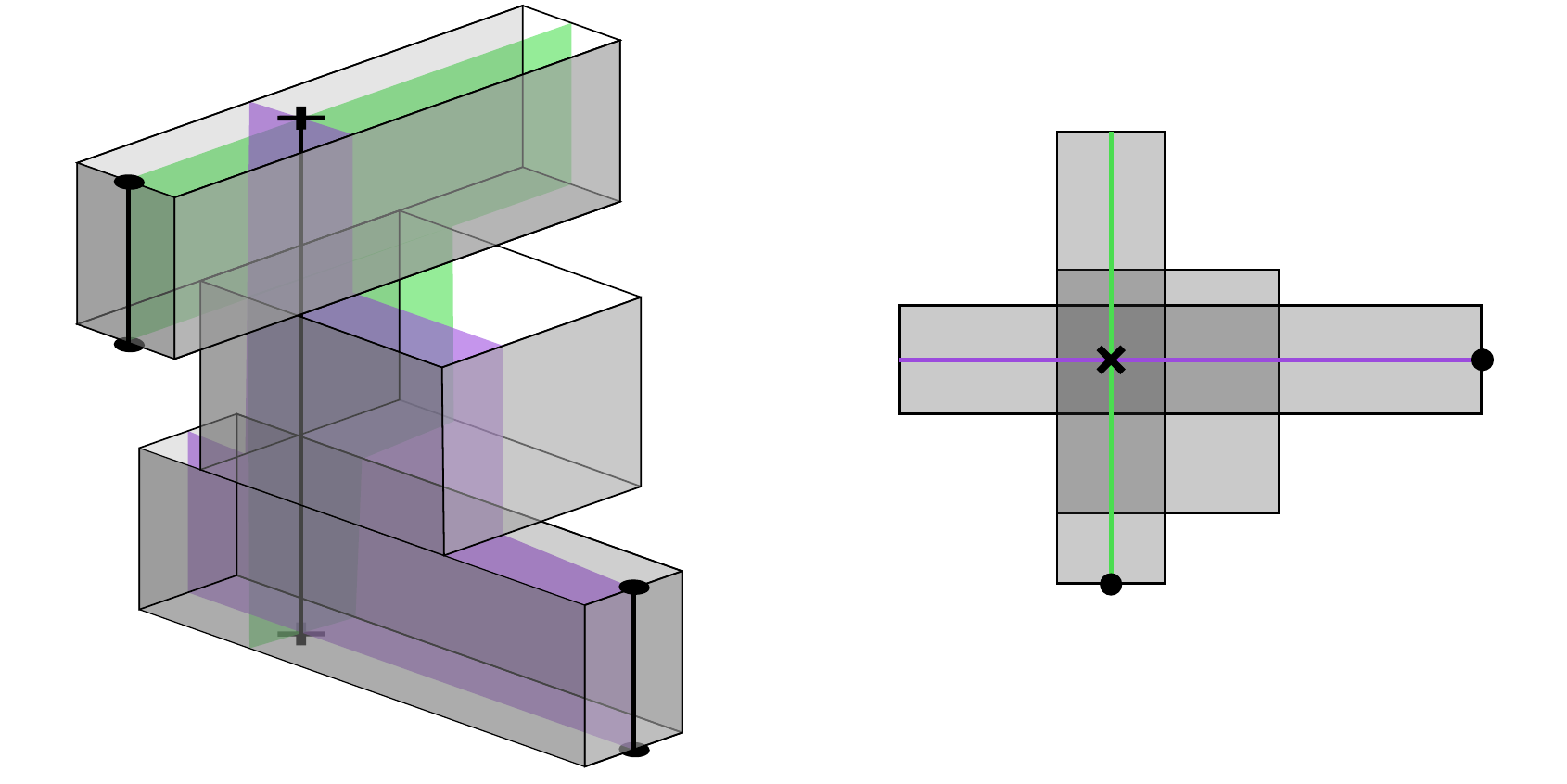}
        \includegraphics[width=0.7\textwidth]{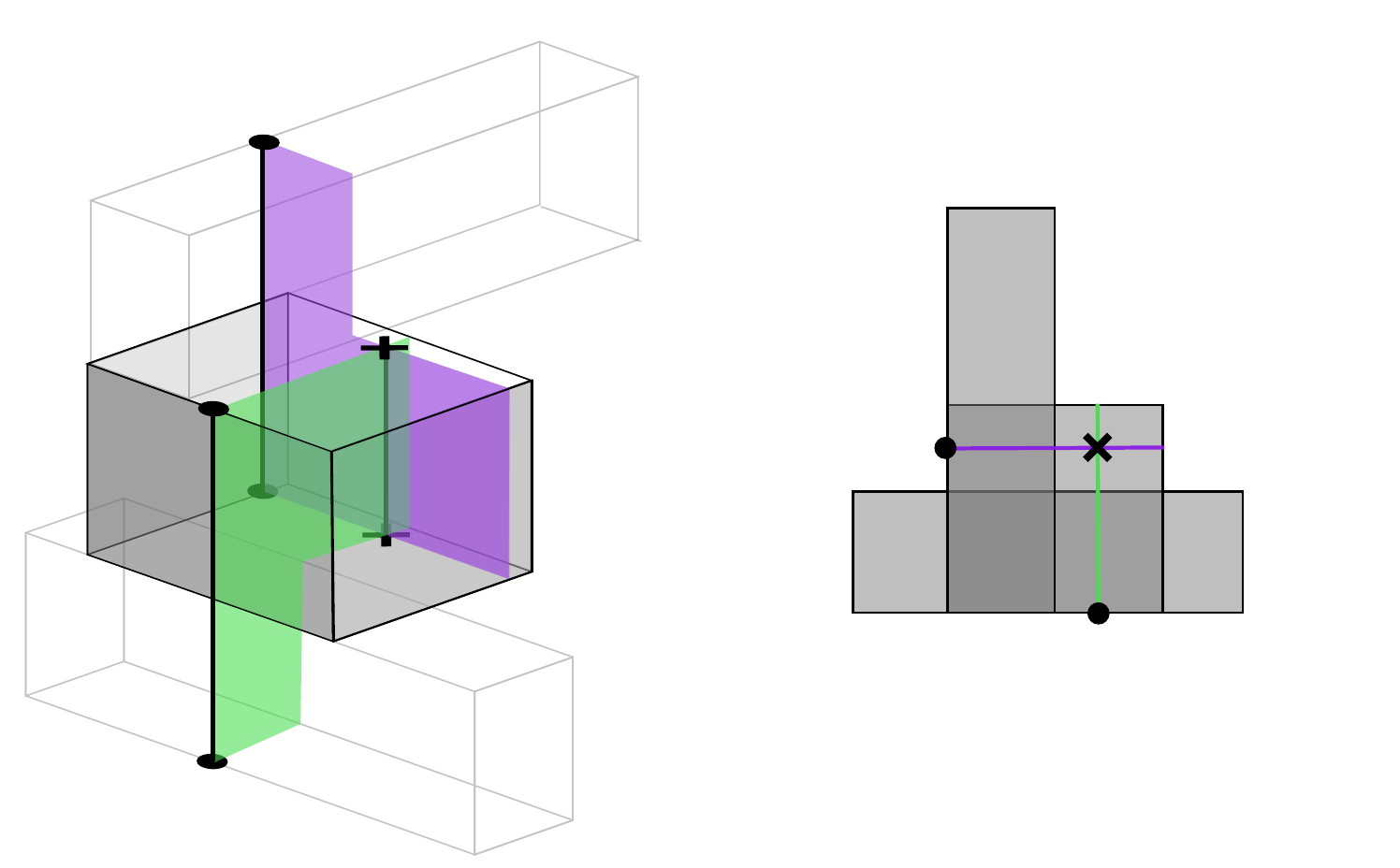}
        \caption{A cornerstone and its shadow in $\mathcal{P}$ in first case (top) and second case (bottom) of \Cref{def:cs}. The unshaded boxes are not included, but show that the chain has finished.}
        \label{fig:cornerstone}
    \end{figure}

    For the other case, suppose that $L$ and $K$ emanate into a common box $B\in\mathcal{W}(\mathcal{A},(B_L,L))\cap \mathcal{W}(\mathcal{A},(B_K,K))$. Now walk upward through the boxes that $K$ emanates into until we reach the maximal box $B_K$ (with respect to the partial order on $\mathcal{A}$) which $K$ emanates into. This will terminate because the boxes must eventually stretch in the corresponding transverse directions due to the Markov property.  We do similarly for $L$, now walking downward, to obtain $B_L$. The resulting chain between $B_L$ and $B_K$ is what we define as $\cs(L,K)$ in this case. See \Cref{fig:cornerstone}, below. We call $\cs(L,K)$, paired with the data of which faces $L$ and $K$ emanate into, the cornerstone of $(L,K)$. Note that $\cs(L,K)$ may not be contained in $\mathcal{A}$, but can be computed once we know that $(L,K)\in\ints(\mathcal{A})$.
\end{con}
A cornerstone is the foundation of where two walls (of boxes) meet. While the construction is for pairs in $\ints(\mathcal{A})$, observe that the cornerstone canonically is canonically associated to $\mathcal{B}$ and a pair $(L,K)\in N^s\times N^u$ of leaves that intersect. Such leaves $L$ and $K$ trace out a path between their underlying cusps in $\mathcal{P}$. In the first case of the definition, the cornerstone is the minimal chain of boxes that projects to cover this path. See \Cref{fig:cornerstone}. The second case is adapted to retain uniqueness. This canonical association allows us to talk freely of cornerstones for pairs of leaves on $N$.

\begin{rem}
Let $e\subset\mathcal{P}$ be an edge rectangle. Observe that $e$ has two cornerstones, one from each of its associated joints. Similarly, cornerstones give a canonical set of boxes for each skeletal rectangle of $\mathcal{P}$.
\end{rem}

We talk of one cornerstone being a \emph{$\pi_1(M^{\circ})$-translate} of another if the chains of boxes are translates in a way that also translates the recording data of which walls the leaf emanate into.

\begin{lemma}\label{lem:cornerstones-char-ints}
    Let $(L,K),(L',K')\in N^s\times N^u$ be pairs of intersecting cusp leaves. Then the pairs $(L,K)$ and $(L',K')$ are $\pi_1(M^{\circ})$-translates if and only if their cornerstones are.
\end{lemma}
\begin{proof}
Suppose that $(L,K) = g(L',K')$ for $g\in\pi_1(M^{\circ})$. Then $\cs(L,K) = g\cdot \cs(L',K')$ by the minimality in the construction of a cornerstone. Moreover, $L$ and $K$ must emanate into the same walls as $gL'$ and $gK'$, respectively. So, their cornerstones are translates.
 
 Conversely, suppose that there is $g\in\pi_1(M^{\circ})$ such that $\cs(L,K) = g\cdot \cs(L',K')$. Then there is a box $B$ in $\cs(L,K)$ such that $L$ emanates into $B$ along and $gL'$ emanates into $B$. Since the cornerstones are translates, we also have that $L$ and $gL'$ emanate into the same wall as $L$. So, $L=gL'$. Similarly, $K=gK'$.
\end{proof}

To apply \Cref{lem:cornerstones-char-ints} we need to compare cornerstones. Note that the projection $P(\cs(L,K))\subset \mathcal{B}_M$ in the original box decomposition is a tuple of boxes. This tuple, and the projections of which walls the cusps emanate from, determine the cornerstone up to translation.

For a face rectangle $f$, we choose a particular cornerstone. There are two cornerstones associated to $\uspan(f)$. Of these two, choose the cornerstone corresponding to the joint $(L,K)$ where $K$ is a leaf for the edge rectangle of $f$ that does not span $f$ in either direction. Call this the \emph{identifying cornerstone for $f$}. 
\begin{lemma}\label{lem:translates-criterion}
Skeletal rectangles are determined up to translation as follows:
    \begin{itemize}
        \item Let $e_1,e_2\subset\mathcal{P}$ be a pair of edge rectangles. Then $e_1$ and $e_2$ are $\pi_1(M^{\circ})$-translates if and only if either one of their cornerstones are $\pi_1(M^{\circ})$-translates.
        
        \item Let $f_1,\,f_2\in \mathcal{P}$ be a pair of face rectangles. Then $f_1$ and $f_2$ are $\pi_1(M^{\circ})$-translates if and only if their identifying cornerstones are translates.
    
    \item Let $r_1,r_2\subset\mathcal{P}$ be a pair tetrahedron rectangles. Then $r_1$ and $r_2$ are $\pi_1(M^{\circ})$-translates if and only if $\sspan(r_1)$ is a $\pi_1(M^{\circ})$-translate of $\sspan(r_2)$ and $\uspan(r_1)$ is a $\pi_1(M^{\circ})$-translate of $\uspan(r_2)$.
    
    \end{itemize}
\end{lemma}
\begin{proof}
We go through the cases. For edge rectangles, note that an edge rectangle is determined by either pair of intersecting leaves in $N^s\times N^u$ that bound it. So, if two edge rectangles share a cornerstone after translation, \cref{lem:cornerstones-char-ints} implies the edge rectangles themselves must be translates.

For face rectangles, suppose that $f_1$ and $f_2$ have identifying cornerstones that are translates. Then $\uspan(f_1)$ and $\uspan(f_2)$ are translates. Since each rectangle $f_i$ is obtained by extending $\uspan(f_i)$ along one of its unstable edges, there are only two choices for each face up to translation. The definition of the identifying cornerstone distinguishes these.

An edge rectangle spans at most one tetrahedron rectangle in each direction. So, a translation of an edge rectangle induces, and is induced by, a translation of the tetrahedron rectangles they span.
\end{proof}

With the criteria of \Cref{lem:cornerstones-char-ints} and \Cref{lem:translates-criterion}, we can identify translates by filtering first through the cusp leaves and cornerstones, then through $\edges(\mathcal{A})$, and finally through $ \faces(\mathcal{A}) $ and $\tets(\mathcal{A})$.

\begin{rem}\label{rem:orientation_fix}
We now adapt this to the setting where the $\pi_1(M^{\circ})$-action reverses the orientations of the invariant foliations. We only need to alter how we deal with leaves that emanate into walls of boxes.

Suppose that $L\in N^s$ emanates into a wall of a box of $\mathcal{A}$. We now define a stable cusp leaf of $\mathcal{A}$ associated to $L$ only when there are two boxes $A_L,B_L\in \max_i(B_i)$ for which $L$ emanates into the east side of $A_L$ and into the west side of $B_L$. Define the cusp leaf now as the tuple $(\{A_L,B_L\},L)$.

For the remaining objects, we carry this forward as follows. We define the wall about this cusp leaf $\mathcal{W}(\{A_L,B_L\},L)$ to now consist of two chains of boxes, one lying below each of $A_L$ and $B_L$. For $K\in N^u$, we call $(L,K)$ a joint when both chains in the wall of $L$ each intersect (if there are two) both chains in the wall about $K$. The skeletal rectangles are defined the same, but with this notion of joint. Then, cornerstones will consist of a pair of chains of boxes, each defined similarly to \cref{def:cs}. We use these cornerstones to identify translates similarly to how we have above.
\end{rem}

\subsection{Finding a veering triangulation}
We now give the other half of the main algorithm. Through \Cref{subsec:find_rects} and \Cref{subsec:translates}, we have computed from $\mathcal{A}$:
\begin{itemize}
    \item the sets of skeletal rectangles shaded by $\mathcal{A}$, as well their subrectangle and veering data from \Cref{con:data},
    
    \item the translate data of the skeletal rectangles shaded by $\mathcal{A}$.
\end{itemize}

So, we can apply \Cref{con:ag} to $\tets(\mathcal{A})$ and build $\mathcal{V}(\tets(\mathcal{A}))$. This gives us the simple $\mathtt{FindVeering}$ routine. We condense together the check of whether tetrahedra with face pairings form a legitimate three-manifold ideal triangulation, and then whether said triangulation is veering.

\begin{algorithm}[ht]
		 \caption{ $\texttt{FindVeering}(\mathcal{A})$}\label{rou:find_veering}
		\begin{algorithmic}[1]
                \If{$\mathcal{V}(\tets(\mathcal{A}))$ is a veering triangulation}
                \State \Return \texttt{True}
                \Else
                \State \Return \texttt{False}
                \EndIf
                
    \end{algorithmic}
\end{algorithm}

In \Cref{alg:has_fits}, we use $\mathtt{FindVeering}$ as a routine by applying it to the ball of boxes with increasing radius.
\begin{prop}\label{lem:find-veering-works}
    The pseudo-Anosov flow with marked orbits $(\varphi,\sing)$ has no genuine perfect fits if and only if there is $n>0$ such that ${\mathtt{FindVeering}}(\mathcal{B}(n))$ returns $\mathtt{True}$.
\end{prop}
\begin{proof}
    This is a direct consequence of \Cref{lem:pre-find-veering-works}.
\end{proof}

\begin{rem}\label{rem:filling_slopes}
    Suppose from the veering triangulation $\mathcal{V}$, we wish to recover the filled flow $\varphi$ as in the veering-flow correspondence. Each boundary component of $M^{\circ}$ has a natural \emph{flow meridian}: the slope that bounds a disk in $M$. It is straightforward to realise the flow meridians as loops in the cell structure on $\partial \overline{{M}^{\circ}}$ given by truncating $\mathcal{B}^{\circ}$. We want to realise the meridians embedded in the boundary triangulation induced by truncating $\mathcal{V}$, since this would allow us to specify a filling slope on $\partial \mathcal{V}$ that will recover $\varphi$.
    
    From our construction, for each edge $e$ of $\mathcal{V}$, we have the pair of cornerstones canonically associated to $e$. These cornerstones give the homotopy data of $e$, since it specifies which cusps are at its endpoints in $\mathcal{P}$. We can embed each edge inside (making a choice) one the projection of boxes of one of these cornerstones. We then use which trios of edges bound faces to fill in disks. This gives a homotopy copy of the two-skeleton of $\mathcal{V}$ embedded within $\mathcal{B}^{\circ}$, which tells us how the flow meridians lie with respect to the boundary triangulation induced by $\mathcal{V}$.
\end{rem}

\section{Main algorithms}\label{sec:final-algos}
\subsection{Recognising perfect fits}
Now we give our main algorithm \texttt{HasPerfectFits}. This solves the decision problem of whether a flow has perfect fits. Recall the two constituent routines $\mathtt{FindFit}$ (\Cref{rou:find_fit}) and $\mathtt{FindVeering}$ (\Cref{rou:find_veering}). We apply $\mathtt{FindVeering}$ to a pseudo-Anosov flow $\varphi$ by applying it to the pair $(\varphi,\sing)$ for $\sing$ the set of singular orbits. This will verify if $\varphi$ has no perfect fits as discussed in \Cref{rem:fits_vs_genuinefits}.

\begin{algorithm}[ht]
		 \caption{ $\mathtt{HasPerfectFits}(\mathcal{B})$
   \label{alg:has_fits}}
		\begin{algorithmic}[1]
                \State $n := 0$
                \While{\texttt{True}}
                \If{$\texttt{FindFit}(n,\mathcal{B})=\mathtt{True}$}
                \State \Return \texttt{True}
                \ElsIf{$\texttt{FindVeering}(\mathcal{B}(n)) = \mathtt{True}$}
                \State \Return \texttt{False}
                \EndIf
                \State $n:=n+1$
                \EndWhile
    \end{algorithmic}
\end{algorithm}
\mainthm
\begin{proof}
     By the two characterising results for each subroutine, \Cref{lem:find_fits_works} and \Cref{lem:find-veering-works}, $\mathtt{HasPerfectFits}$ applied to the given box decomposition $\mathcal{B}$ terminates and terminates correctly.
\end{proof}

Now consider the generalised case, where $\varphi$ is a (pseudo-)Anosov flow with marked orbits, and we seek genuine perfect fits (\Cref{def:rel_fits}). The subroutine $\mathtt{FindVeering}$ already applies to this setting, and the $\mathtt{FindFit}$ routine can be adapted to this setting as in \Cref{rem:anosov-case}. So, we have the following generalised statement of the theorem.
\begin{coro}\label{thm:main-ext}
    There is an algorithm that, given $\mathcal{B}$ a box decomposition of a (pseudo-)Anosov flow with marked orbits, decides if the flow has genuine perfect fits.
\end{coro}

\subsection{Recognising flows}\label{subsec:recog_flows}
When $\mathtt{HasPerfectFits}(\mathcal{B})$ returns false, we use the resulting veering triangulation to decide the equivalence problem. Let $\psi$ be another pseudo-Anosov flow on a manifold $N$. We say that $\varphi$ and $\psi$ are \emph{orbit equivalent} if there exists a homeomorphism from $M$ to $N$ carrying orbits of $\varphi$ to orbits of $\psi$. Suppose that we also have a box decomposition $\mathcal{A}$ of $\psi$. If $\mathtt{HasPerfectFits}$ returns false for both $\mathcal{A}$ and $\mathcal{B}$, we have veering triangulations $\mathcal{V}(\varphi)$ and $\mathcal{V}(\psi)$.

Note that if $\varphi$ and $\psi$ are orbit equivalent, this induces an equivariant homeomorphism on the level of their drilled orbit spaces, so $\mathcal{V}(\varphi)$ and $\mathcal{V}(\psi)$ are isomorphic. We allow our veering isomorphisms to reverse the transverse structures; this allows us to consider orbit equivalences which reverse the orientations of orbits on the flow.

 By the remaining promised step of Schleimer-Segerman's program \cite{ss-final-step}, we have the converse: the two flows will be equivalent if and only if both the veering triangulations and their filling slopes agree.  Now, given $\varphi$ and $\psi$, we can first compare their veering triangulations. As finite combinatorial structures, it is straightforward to determine when two veering triangulations are isomorphic. This can even be done in practice using \emph{Veering} \cite{Veering}. Then, we check if the isomorphism between the veering triangulations preserves the filling slopes.

\equiv

We again have an analogue of \Cref{cor:equivalence} in the marked case. We say that two marked pseudo-Anosov flows $(\varphi,\sing)$ and $(\varphi',\{s'_i\})$ are \emph{orbit equivalent} if there is an orbit equivalence of the underlying flows that induces a (possibly permuting) bijection between the sets $\sing$ and $\{s'_i\}$. If two such flows $(\varphi,\sing)$ and $(\varphi',\{s'_i\})$ have no genuine perfect fits, we can determine whether the flows are orbit equivalent.

When one allows the input to be genuine Anosov flows, the same algorithm can also be used to check if two flows are \emph{almost equivalent} through surgery along a set of specified orbits from each flow. See \cite[\S 1]{tsang-friedconjecture} for a discussion of the history and importance of this operation.

\subsection{Recognising suspensions}
We give a final application. For $\varphi$ to be a suspension flow, it must have no perfect fits, so our algorithm will return its veering triangulation $\mathcal{V}(\varphi)$. It is further necessary for $\mathcal{V}(\varphi)$ to be \emph{layered} \cite[Definition 2.16]{ss-link} as per Agol's original construction \cite{agol-veering}. Let $\Sigma\subset M$ be a surface carried by the two-skeleton of $\mathcal{V}$. Suppose that $\Sigma$ is a fiber, that is, $M-\Sigma\simeq \Sigma \times I$. The surface $\Sigma$ induces a (multi)slope on $\partial \mathcal{V}$; call this a \emph{fiber slope} of $\mathcal{V}$. Since $\mathcal{V}(\varphi)$ is layered, it specifies a fibered face of the Thurston norm ball (see \cite[Theorem E]{lmt-polynomial}). All possible fibers $\Sigma$ correspond to irreducible integer classes in the cone over this face.

  Recall from \Cref{rem:filling_slopes} that we have the (multi)slope $\mu({\varphi})$ of the flow meridian(s); this is an element of $H_1(\partial M)\simeq \bbZ^{2n}$ where $n$ is the number of boundary components of $M^{\circ}$. Observe that the flow $\varphi$ is a suspension flow if only if filling back the original meridian disk caps off a disk in some choice of fibering of $M^{\circ}$. That is, precisely when $\mu({\varphi})$ is a fiber slope for some choice of fiber. To determine when this occurs, we must compute the set of fiber slopes. This means computing the fibered face corresponding to $\mathcal{V}(\varphi)$; doing so has been implemented in the \emph{Veering} software \cite{Veering}. They compute the rays that bound the cone in $H_2(M,\partial M)$. The fiber slopes are the image in $H_1(\partial M)\simeq \bbZ^{2n}$ of the linear map $H_2(M,\partial M)\to H_1(\partial M)$. So, this image is also a cone. Then, checking whether $\varphi$ is a suspension amounts to checking if $\mu({\varphi})$ satisfies the corresponding linear inequalities.

\begin{coro}\label{cor:suspension-recognition}
    There is an algorithm to decide, given a box decomposition of a pseudo-Anosov flow, if the flow is a suspension.
\end{coro}
For a chosen fiber, one can also compute the corresponding monodromy of the flow using the techniques implemented in Bell's \emph{Flipper} software \cite{flipper}.
\section{Questions}\label{sec:qns}
We start with the problem of promoting either of the constituent routines of the $\mathtt{HasPerfectFits}$ algorithm (\Cref{alg:has_fits}) to a complete algorithm.

Consider the \texttt{FindFit} routine (\Cref{rou:find_fit}). Suppose that $\varphi$ has perfect fits. So, as in the characterisation \Cref{thm:pfits-iff-homotopy}, $\varphi$ has a pair of (possibly identical) almost freely homotopic orbits $(\alpha,\beta)$. We define the length of $(\alpha,\beta)$ as the sum of the length of their itineraries. Let $\vert \mathcal{B}\vert$ be the number of edges in $\mathcal{M}(\mathcal{B})$.

\begin{question}\label{ques:homotopies}
     Does there exist a uniform bound in terms of $\vert \mathcal{B}\vert $ on the shortest length of a pair of almost freely homotopic orbits $(\alpha,\beta)$?
\end{question}
The closest to a result of this flavour is given by Tsang in \cite[Theorem 6.4]{tsang-birkhoff}. They give bounds on the length of orbits needed to be drilled to produce a Birkhoff section, but this requires no perfect fits from the outset.

To upgrade the \texttt{FindVeering} routine (\Cref{rou:find_veering}), we need to know when all edge rectangles have been found. One such way is with cornerstones.

\begin{question}\label{ques:frames}
     Suppose that $\varphi$ has no perfect fits. Does there exist a uniform bound, in terms of $\vert \mathcal{B}\vert $, on the size of cornerstones of edge rectangles?
\end{question}
 Note that since cusp leaves are dense in $M$, there are always arbitrarily large cornerstones for individual pairs of leaves as opposed to edge rectangles.
 
Next, we state some natural questions that would suggest alternative approaches to the main problem.
\begin{question}\label{ques:tets}
    Does there exist $C>0$ so that for all pseudo-Anosov flows $\varphi$ without perfect fits we have $\vert \mathcal{V}(\varphi)\vert \leq C \vert \mathcal{B} \vert +C ?$
\end{question}
The number of boxes alone does not control tetrahedra: an Anosov suspension always has a box decomposition with two boxes, while the number of tetrahedra can be unbounded. A converse bound to \Cref{ques:tets} is impossible, since we can subdivide boxes arbitrarily.

These questions are closely related to the quasigeodesic properties of the underlying flow. Recall that $\varphi$ is \emph{quasigeodesic} if there is $q>0$ such that the flow lines in the universal cover are $q$-quasigeodesics. For a general reference, see \cite[\S 10]{calegari-folns-book}. Fenley \cite[Main Theorem]{fenley-qg} shows that $\varphi$ is quasigeodesic when it is \emph{bounded}, which includes when $\varphi$ has no perfect fits. Endow $M$ with a metric by giving boxes height one and aspect ratios following \cite[Lemma 3.1.1]{mosher1996laminations}. The metric does not necessarily match across gluings, so we only obtain a path metric; see \cite[page 3440]{agol-tsang} for a thorough description. Following \Cref{rem:projecting-is-hard}, knowing the quasigeodesic constant for $\varphi$ would allow us to project $\mathcal{B}$ to $\mathcal{P}$, and allow for a different approach to the \texttt{FindVeering} routine.
\begin{question}\label{qn:quasigeodesic}
    Suppose that $\varphi$ has no perfect fits. Is there an upper bound on the quasigeodesic constant $q$ in terms of $\vert \mathcal{B}\vert$?
\end{question}

We finish with the natural question that follows \Cref{cor:equivalence}.
\begin{question}
    Is the orbit equivalence problem for transitive pseudo-Anosov flows solvable?
\end{question}
One can ask the same question when considering orbit equivalence for unmarked flows. We have access to powerful tools through Anosov-like actions \cite{bfm-classi} and veering triangulations (say, in the form of \cite[Theorem 6.1]{tsang-friedconjecture}). However, one needs a robust way to compare the nonsingular periodic orbits between flows. Another direction could be the new classifying invariant of \emph{geometric types} introduced in \cite{geometric-types}, where some similar algorithmic problems are considered.
\printbibliography
\end{document}